\documentclass[11pt]{amsart}
\usepackage{amsmath,amssymb, amsthm,delarray}
\usepackage{graphicx}
\numberwithin{equation}{section}
\newtheorem{thm}{Theorem}[section]
\newtheorem{lemma}[thm]{Lemma}
\newtheorem{remark}[thm]{Remark}
\newtheorem{corollary}[thm]{Corollary}
\newtheorem{proposition}[thm]{Proposition}
{\rm}
\newtheorem{example}{Example}{\rm}
\newtheorem{assumption}{Assumption}{\rm}

\def\beq{\begin{equation} }
\def\eeq{\end{equation} }

\def\C{\mathbb{C}}

\def\N{\mathbb{N}}
\def\M{\mathbf{M}}
\def\R{\mathbb{R}}
\def\A{\mathbf{A}}
\def\P{\mathbf{P}}

\def\K{\mathbf{K}}
\def\B{\mathbf{B}}

\def\M{\mathbf{M}}

\def\G{\mathbf{G}}
\def\H{\mathbf{H}}

\def\v{\mathbf{v}}
\def\f{\mathbf{f}}

\def\u{\mathbf{u}}
\def\f{\mathbf{f}}

\def\b{\mathbf{b}}
\def\a{\mathbf{a}}
\def\c{\mathbf{c}}
\def\x{\mathbf{x}}
\def\g{\mathbf{g}}

\def\y{\mathbf{y}}
\def\z{\mathbf{z}}

\def\dis{\displaystyle}

\def\R{\mathbb{R}}
\def\B{\mathbf{B}}
\def\C{\mathbb{C}}

\def\vol{{\rm vol}\,}
\def\dis{\displaystyle}

\begin{document}
\title{A generalization of L\"owner-John's ellipsoid theorem}
\author{Jean B. Lasserre}
\address{LAAS-CNRS and Institute of Mathematics\\
University of Toulouse,
7 Avenue du Colonel Roche, 310777 Toulouse Cedex 4, France.}
\email{lasserre@laas.fr}

\begin{abstract}
We address the following generalization $\P$ of the L\"owner-John ellipsoid problem.
Given a (non necessarily convex) compact set $\K\subset\R^n$ and an even integer $d\in\N$, 
find an homogeneous
polynomial $g$ of degree $d$ 
such that $\K\subset \G:=\{\x:g(\x)\leq1\}$ and $\G$ has minimum volume among all such sets.
We show that $\P$ is a convex optimization problem even if neither $\K$ nor $\G$
are convex! We next show that $\P$ has a unique optimal solution 
and a characterization with at most ${n+d-1\choose d}$  contacts points in
$\K\cap\G$ is also provided. This is the analogue for $d>2$ of the L\"owner-John's theorem in
the quadratic case $d=2$, but importantly,  we neither require the set $\K$ nor the sublevel set $\G$ to be convex.
More generally, there is also an homogeneous polynomial $g$ of even degree $d$ and a point $\a\in\R^n$ such that
$\K\subset \G_\a:=\{\x:g(\x-\a)\leq1\}$ and $\G_\a$ has minimum volume among all such sets
(but uniqueness is not guaranteed). Finally, we also outline a numerical scheme to approximate 
as closely as desired the optimal value and an optimal solution. It consists of solving a hierarchy of convex optimization 
problems with strictly convex objective function and Linear Matrix Inequality (LMI) constraints.
\end{abstract}


\maketitle

\section{Introduction}

``Approximating" data by relatively simple geometrical objects is a fundamental problem
with many important applications and 
the ellipsoid of minimum volume is the most  well-known of the associated computational techniques.

In addition to its nice properties from the viewpoint of applications,
the ellipsoid of minimum volume is also very interesting from a mathematical viewpoint.
Indeed, if $\K\subset\R^n$ is a convex body, computing an ellipsoid of minimum volume 
that contains $\K$
is a classical and famous problem which has a unique optimal solution called the {\it L\"owner-John} ellipsoid. In addition, John's theorem states that the optimal ellipsoid $\Omega$ is characterized by
$s$ contacts points $\u_i\in\K\cap\Omega$ (more precisely $\u_i\in\partial\K\cap\partial\Omega$), and positive scalars $\lambda_i$, $i=1,\ldots,s$,
where $s$ is bounded above by $n(n+3)/2$ in the general case and $s\leq n(n+1)/2$ when
$\K$ is symmetric; see e.g. Ball \cite{ball1,ball2}, Henk \cite{henk}.  More precisely,
the  unit ball $B_n$ is the unique ellipsoid with minimum volume containing $\K$ if and only if
$\sum_{i=1}^s\lambda_i\u_i=0$ and $\sum_{i=1}^s\lambda_i\u_i\u_i^T=I_n$,
where $I_n$ is the $n\times n$ identity matrix.

In particular, and in contrast to other approximation techniques,
computing the ellipsoid of minimum volume is a {\it convex} optimization problem for which efficient techniques
are available; see e.g. Calafiore \cite{calafiore} and Sun and Freund \cite{sun} for more details. 
For a nice recent historical survey on the L\"owner-John's ellipsoid, the interested reader is referred to the recent paper by Henk \cite{henk} and the many references therein.

As underlined in Calafiore \cite{calafiore}, {\it ``The problem of approximating observed data
with simple geometric primitives is, since the time of Gauss, of fundamental importance
in many scientific endeavors"}. For practical purposes and numerical efficiency
the most commonly used are polyhedra and ellipsoids and such techniques are ubiquitous in several different area,
control, statistics, computer graphics, computer vision, to mention a few. For instance:

$\bullet$ In {\it robust linear control}, one is interested in outer or inner
approximations of the stability region associated with a linear dynamical system, that is, the set of initial states from which the system can be stabilized by some control policy. Typically, the stability region which can be formulated as a semi-algebraic 
set in the space of coefficients of the characteristic polynomial, is non convex.
By using the Hermite stability criterion, it can be described by a parametrized polynomial matrix inequality
where the parameters account for uncertainties and the variables are the controller coefficients.
Convex inner approximations of the stability region
have been proposed in form of polytopes in Nurges \cite{polytopes},  ellipsoids in Henrion et al. \cite{ellipsoids},
and more general convex sets defined by Linear Matrix Inequalities (LMIs) in Henrion et al. \cite{lmi1}, and Karimi et al. \cite{lmi2}.

$\bullet$ In {\it statistics} one is interested in the ellipsoid $\xi$ of minimum volume covering some given 
$k$ of $m$ data points because $\xi$  has some interesting statistical properties such as 
affine equivariance and positive breakdown properties \cite{croux}. In this context 
the center of the ellipsoid is called the minimum volume ellipsoid (MVE) 
{\it location} estimator and the associated matrix associated with $\xi$ is
called the MVE {\it scatter} estimator; see e.g. Rousseeuw \cite{rousseeuw} and Croux et al. \cite{croux}.

$\bullet$  In {\it pattern separation},
minimum volume ellipsoids are used for separating two sets of data points. For computing
such ellipsoids, convex programming techniques have been used in the early work of Rosen \cite{rosen} and
more modern semidefinite programming techniques in Vandenberghe and Boyd \cite{boyd}. Similarly,
in robust statistics and data mining the ellipsoid of minimum
volume covering a finite set of data points identifies {\it outliers}
as the points on its boundary; see e.g. Rousseeuw and Leroy \cite{rousseeuw}. 
Moreover, this ellipsoid technique is also scale invariant, a highly desirable property in data mining which is
not enjoyed by other clustering methods based on various {\it distances}; see the discussion in 
Calafiore \cite{calafiore}, Sun and Freund \cite{sun} and references therein.

$\bullet$ Other clustering techniques in {\it computer graphics}, {\it computer vision} and {\it pattern recognition}, 
use various (geometric or algebraic) distances (e.g. the equation error) and 
compute the best ellipsoid by minimizing an associated non linear least squares
criterion (whence the name ``least squares fitting ellipsoid" methods). 
For instance, such techniques have been proposed in computer graphics and computer vision by
Bookstein \cite{bookstein} and Pratt \cite{pratt}, in pattern recognition  by Rosin \cite{rosin}, Rosin and West \cite{rosin2},
Taubin \cite{taubin}, and in
another context by Chernousko \cite{chernousko}. When using
an algebraic distance (like e.g. the equation error) the geometric interpretation
is not clear and the resulting ellipsoid may not be satisfactory; see e.g. 
an illuminating discussion in Gander et al. \cite{gander}. 
Moreover, in general the resulting optimization problem is not convex and 
convergence to a global minimizer is not guaranteed. \\

So optimal data fitting using an ellipsoid of minimum volume is not only satisfactory from the viewpoint of applications
but is also satisfactory from a mathematical viewpoint as it reduces to a (often tractable) convex optimization problem
with a unique solution having a nice characterization in term of contact points in $\K\cap\Omega$.
In fact, reduction to solving a convex optimization problem 
with a unique optimal solution, is a highly desirable property of any data fitting technique!

\subsection*{A more general optimal data fitting problem}

In the L\"owner-John problem one restricts to convex bodies $\K$ because
for a non convex set $\K$ the optimal ellipsoid is also solution 
to the problem where $\K$ is replaced with its convex hull ${\rm co}(\K)$. However, if one considers 
sets that are more general than ellipsoids, an optimal solution for $\K$ is not necessarily the same as
for ${\rm co}(\K)$, and indeed, in some applications one is interested in
approximating as closely as desired a non convex set $\K$. In this case
a non convex approximation is sometimes highly desirable as more efficient.

For instance, in the robust control problem already alluded to above,
in Henrion and Lasserre \cite{ieee} we have provided
an inner approximation of the stability region $\K$  by the sublevel set $\G=\{\x: g(\x)\leq0\}$ of
a non convex polynomial $g$. By allowing 
the degree of $g$ to increase one obtains the convergence ${\rm vol}(\G)\to{\rm vol}(\K)$
which is impossible with convex polytopes, ellipsoids and LMI approximations as described in \cite{polytopes,ellipsoids,lmi1,lmi2}.

So if one considers the more general data fitting problem
where $\K$ and/or the (outer) approximating set are allowed to be non convex,
can we still infer interesting conclusions 
as for the L\"owner-John problem? Can we also derive a practical algorithm for computing (or at least approximating)
an optimal solution? 

The purpose of this paper is to provide results in this direction
that can be seen as a non convex generalization of the Lowner-John
problem but,  surprisingly, still reduces to solving a convex optimization problem with a unique optimal solution.

Some works have considered generalizations of the L\"owner-John problem.
Fo instance, Giannopoulos et al. \cite{giannopoulos} have extended John's theorem for couples $(\K_1,\K_2)$ of convex bodies when $\K_1$ is {\it in maximal volume position} of $\K_1$ inside $\K_2$, whereas Bastero and Romance \cite{bastero} refined this result by allowing $\K_1$ to be non-convex.

In this paper 
we consider a different non convex generalization of the 
L\"owner-John ellipsoid problem, with a more algebraic flavor. Namely, we address the following two problems $\P_0$ and $\P$.\\

{\em $\P_0$: Let $\K\subset\R^n$ be a compact set (not necessarily convex) 
and let $d$ be an even integer. 
Find an homogeneous polynomial $g$ of degree $d$ such that its sublevel set
$\G_1:=\{\x\,:\,g(\x)\leq 1\}$ contains $\K$ and has minimum volume among all such sublevel sets
with this inclusion property.}\\

{\em $\P$: Let $\K\subset\R^n$ be a compact set (not necessarily convex)
and let $d$ be an even integer. Find an homogeneous polynomial $g$ of degree $d$ and $\a\in\R^n$ such that the sublevel set
$\G_1^\a:=\{\x\,:\,g(\x-\a)\leq 1\}$ contains $\K$ and has minimum volume among all such sublevel sets
with this inclusion property.}\\

Necessarily $g$ is a nonnegative homogeneous polynomial since otherwise 
the volumes of $\G_1$ and $\G^\a_1$ are not finite. Of course, when $d=2$
then $g$ is convex (i.e., $\G_1$ and $\G^\a_1$ are ellipsoids)
because every nonnegative quadratic form defines a convex function, and $g$ is an optimal solution for problem $\P$ with $\K$ or its convex hull ${\rm co}(\K)$. That is, one retrieves the L\"owner-John problem.
But when $d>2$ then $\G_1$ and $\G_1^\a$ are not necessarily convex. For instance, take
$\K=\{\x:g(\x)\leq 1\}$ where $g$ is some nonnegative homogeneous polynomial such that $\K$ is compact but non convex. Then $g$ is an optimal solution for problem $\P_0$ with $\K$ and cannot be optimal for ${\rm co}(\K)$;  a two-dimensional example is $(x,y)\mapsto g(x,y):=x^4+y^4-\epsilon x^2y^2$
and another one is $g(x,y):=x^6+y^6-\epsilon x^3y^3$, for $\epsilon>0$ sufficiently small;
see Figure \ref{figure-2}.

\begin{center} 
\begin{figure}
\resizebox{0.9\textwidth}{!}
{\includegraphics{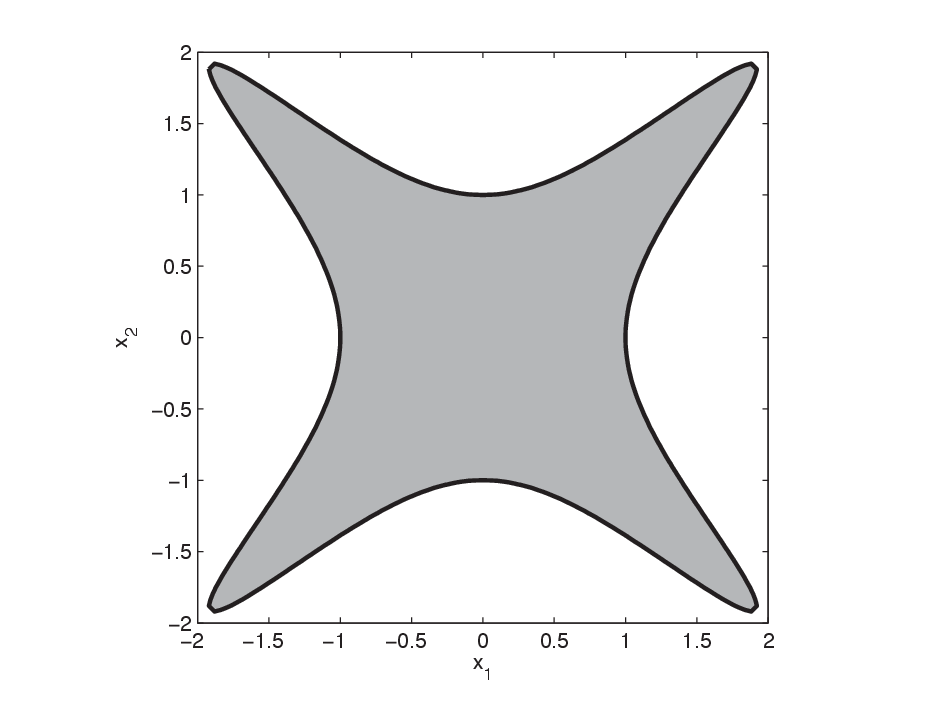}\includegraphics{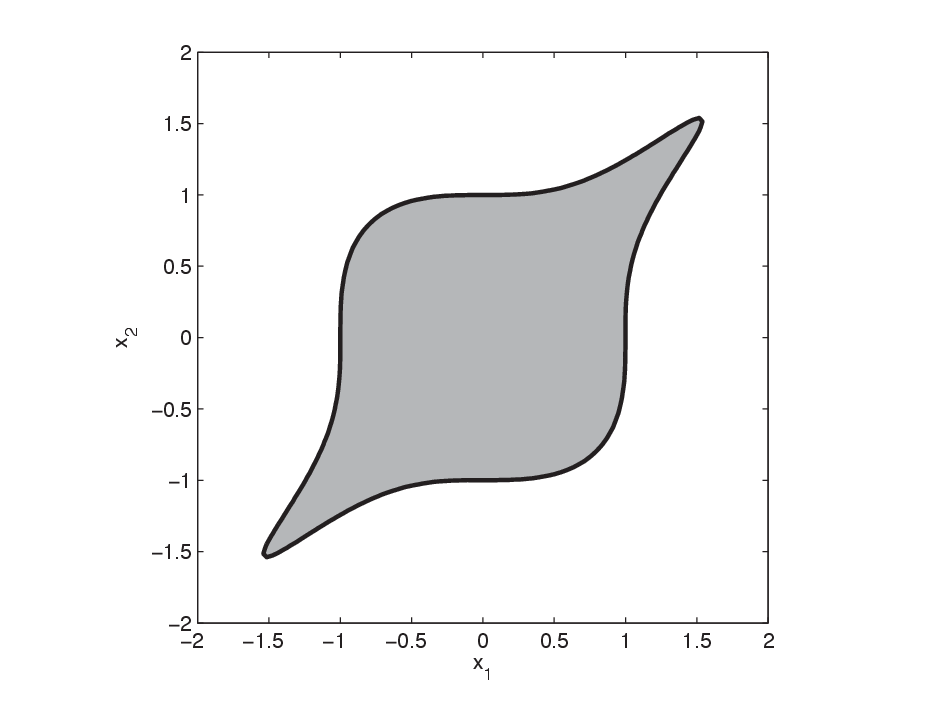}}
\caption{$\G_1$ with $g(\x)=x^4+y^4-1.925\,x^2y^2$ (left) and with $g(\x)=x^6+y^6-1.925\,x^3y^3$ (right)\label{figure-2}}
\end{figure}
\end{center}

\subsection*{Contribution} 
We show that problem $\P_0$ and $\P$ are  indeed natural generalizations of the L\"owner-John ellipsoid problem in the sense that:

- (a) $\P_0$ also has a {\it unique} solution $g^*$.

- (b) A characterization of $g^*$  also involves
$s$ contact points  in $\K\cap\G_1$ (more precisely in $\partial\K\cap\partial\G_1$), where $s$ is now bounded by ${n+d-1\choose d}$ (when $d=2$ one retrieves the bound for the symmetric L\"owner-John problem).

And so when $d=2$ we retrieve the symmetric L\"owner-John problem as a particular case. In fact
it is shown that $\P_0$ is a convex optimization problem
no matter if neither $\K$ nor $\G_1$ are convex. Of course, convexity in itself does not guarantee  a favorable computational complexity\footnote{For instance some well-known NP-hard 0/1 optimization problems reduce to conic LP optimization problems over the
convex cone of copositive matrices (and/or its dual) for which the associated membership problem is hard.}.
As we will see $\P_0$ reduces to minimizing a strictly convex function over a convex cone intersected with an affine subspace and
``hardness" of $\P_0$ is reflected in two of its components: (i) The (convex)  objective function as well as its gradient and Hessian are  difficult to evaluate, 
and (ii) the cone membership problem is NP-hard in general.
However convexity is crucial to show the uniqueness and characterization of the optimal solution in (a) and (b) above.

We use an intermediate and crucial result of independent interest. Namely, 
the Lebesgue-volume function $g\mapsto v(g):={\rm vol}(\G_1)$ is a strictly convex function of the coefficients of $g$, which is far from being obvious from its definition. 
Concerning the more general problem $\P$, we also show that there is an optimal solution
$(g^*,\a^*)$ with again a characterization which involves $s$ contact points in $\K\cap\G^{\a^*}_1$, but now uniqueness is not guaranteed.
Again and importantly, in both problems $\P_0$ and $\P$, neither $\K$ nor $\G_1^\a$ are required to be convex.

\subsection*{On the computational side} Even though $\P_0$ is a convex optimization problem, it is hard to solve because 
even if $\K$ would be a finite set of points (as is the case in statistics applications of the L\"owner-John problem) and
in contrast to the quadratic case,  evaluating the (strictly convex) objective function, its gradient and Hessian 
can be a challenging problem, especially if the dimension is larger than $n=3$. Indeed 
evaluating the objective function reduces to computing the Lebesgue volume of the sublevel set $\G_1$
whereas evaluating its gradient and Hessian requires computing other moments of the Lebesgue measure on $\G_1$.
So this is one price to pay for the generalization of the L\"owner-John ellipsoid problem. (Notice however that if
$\K$ is not a finite set of points then even the L\"owner-John ellipsoid problem
is also hard to solve because for more general sets $\K$ 
the inclusion constraint $\K\subset\xi$ (or ${\rm conv}(\K)\subset\xi$) can be difficult to handle.) 
In general, and even for convex bodies, computing the volume 
is an NP-hard problem; in fact even approximating the volume efficiently  within given bounds is hopeless.
For more details the  interested reader is referred to e.g. Barvinok \cite{barvinok}, Dyer et al. \cite{dyer} and the many references therein. 
On the other hand, even though $\G_1$ is not necessarily convex, it is still a rather specific set and 
assessing a precise computational complexity for its volume remains to be done. 

However, we can still approximate 
as closely as desired the objective function as well as its gradient and Hessian
by using the methodology developed in Henrion et al \cite{sirev}, especially when the dimension 
is small $n=2,3$ (which is the case in several applications in statistics).

Moreover, if $\K$ is a (compact) basic semi-algebraic set with an explicit description
$\{\x: g_j(\x)\geq0,\:j=1,\ldots,m\}$ for some polynomials $(g_j)\subset\R[\x]$, then
we can use powerful positivity certificates from real algebraic geometry
to handle the inclusion constraint $\G_1\supset\K$ in the associated convex optimization problem.
Therefore, in this context, we also provide a numerical scheme to approximate the optimal value and the unique optimal solution of $\P_0$ as closely as desired. It consists of solving a hierarchy of convex optimization problems where each problem
in the hierarchy has a strictly convex objective function and a feasible set defined by Linear Matrix Inequalities (LMIs).

\section{Notation, definitions and preliminary results}

\subsection{Notation and definitions}
Let $\R[\x]$ be the ring of polynomials in the variables $\x=(x_1,\ldots,x_n)$ and let
$\R[\x]_d$ be the vector space of polynomials of degree at most $d$
(whose dimension is $s(d):={n+d\choose n}$).
For every $d\in\N$, let  $\N^n_d:=\{\alpha\in\N^n:\vert\alpha\vert \,(=\sum_i\alpha_i)\leq d\}$, 
and
let $\v_d(\x)=(\x^\alpha)$, $\alpha\in\N^n$, be the vector of monomials of the canonical basis 
$(\x^\alpha)$ of $\R[\x]_{d}$. 
For two real symmetric matrices $\B,\mathbf{C}$, the notation
$\langle \B,\mathbf{C}\rangle$ stands for ${\rm trace}\,(\B\mathbf{C})$; also, the notation $\B\succeq0$ (resp. $\B\succ0$) stands for $\B$ is positive semidefinite (resp. positive definite).
A polynomial $f\in\R[\x]_d$ is written
\[\x\mapsto f(\x)\,=\,\sum_{\alpha\in\N^n}f_\alpha\,\x^\alpha,\]
for some vector of coefficients $\f=(f_\alpha)\in\R^{s(d)}$. A polynomial $f\in\R[\x]_d$ is homogeneous
of degree $d$ if $f(\lambda\x)=\lambda^d f(\x)$ for all $\x\in\R^n$ and all $\lambda\in\R$.

Let us denote by $\H[\x]_d$, $d\in\N$, the space of homogeneous polynomials of degree $d$ and
$\P[\x]_d\subset \H[\x]_d$, its subset of homogeneous polynomials of degree $d$ 
such that their sublevel set $\G_1:=\{\,\x\,:\,g(\x)\leq 1\,\}$ has finite Lebesgue volume, denoted
${\rm vol}(\G_1)$. Notice that $g\in\P[\x]_d$ is necessarily nonnegative
(so that $d$ is necessarily even) and $0\not\in\P[\x]_d$; but $\P[\x]_d$
is {\it not} the set of positive semidefinite (psd) forms of degree $d$ (excluding the zero form); indeed
if $n=2$ and $\x\mapsto g(\x)=(x_1-x_2)^2$ then $\G_1$ does not have finite Lebesgue volume.
On the other hand when $g\in\P[\x]_d$ the set $\G_1$ is not necessarily bounded;
for instance if $n=2$ and $\x\mapsto g(\x):=x_1^2x_2^2(x_1^2+x_2^2)$,
the set $\G_1$ has finite volume but is not bounded\footnote{We thank Pham Tien Son for providing these
two examples.}. So $\P[\x]_d$ is not the space of positive definite (pd) forms of degree $d$ either.

For $d\in\N$ and a closed set $\K\subset\R^n$, denote by $C_d(\K)$ the convex cone of 
all polynomials of degree at most $d$ that are nonnegative on $\K$, and denote by
$\mathcal{M}(\K)$ the Banach space of finite signed Borel measures with support contained in $\K$
(equipped with the total variation norm). Let $\mathcal{M}(\K)_+\subset\mathcal{M}(\K)$ be the convex cone
of finite (positive) Borel measures on $\K$.

In the Euclidean space $\R^n$ we denote by $\langle\cdot,\cdot\rangle$ the usual duality bracket.
 
 \subsection*{Laplace transform}
 Given a measurable function $f:\R\to\R$ with $f(t)=0$ for all $t<0$, its one-sided (or unilateral) {\it Laplace transform}
 $\mathcal{L}[f]:\C\to\C$ is defined by
 \[\lambda\mapsto\quad \mathcal{L}[f](\lambda)\,:=\,\int_0^\infty \exp(-\lambda t)f(t)\,dt,\qquad \lambda\in\,D,\]
where its domain $D\subset\C$ is the set of $\lambda\in\C$ where the above integral is finite. 
For instance, let $f(t)=0$ if $t<0$ and $f(t)=t^a$ for $t\geq0$ and $a>-1$. Then
$\mathcal{L}[f](\lambda)=\frac{\Gamma(a+1)}{\lambda^{a+1}}$ and $D=\{\lambda:\Re(\lambda)>0\}$.
Moreover
$\mathcal{L}[f]$ is analytic on $D$ and therefore if there exists an analytic function $F$ such that
$\mathcal{L}[f](\lambda)=F(\lambda)$ for all $\lambda$ in a segment of the real line contained in $D$,
then $\mathcal{L}[f](\lambda)=F(\lambda)$ for all $\lambda \in D$. This is a consequence of the Identity Theorem for analytic functions;
see e.g. Freitag and Busam \cite[Theorem III.3.2, p. 125]{freitag}. A classical reference for the Laplace transform
is Widder \cite{widder}.

 \subsection{Some preliminary results}
 
 We first  have the following result:
 \begin{lemma}
 \label{prop1}
 The set $\P[\x]_d$ is a convex cone.
 \end{lemma}
 \begin{proof}
 Let $g,h\in\P[\x]_d$ with associated sublevel sets
 $\G_1=\{\x:g(\x)\leq1\}$ and $\H_1=\{\x:h(\x)\leq1\}$. For $\lambda\in (0,1)$, consider the nonnegative homogeneous polynomial $\theta:=\lambda g+(1-\lambda)h\in\R[\x]_d$,
 with associated sublevel set 
 \[\Theta_1\,:=\,\{\x\::\:\theta(\x)\,\leq \,1\}\,=\,\{\x\::\:\lambda g(\x)+(1-\lambda)h(\x)\,\leq\,1\}.\]
 Write $\Theta_1=\Theta_1^1\cup\Theta_1^2$ where 
 $\Theta_1^1=\Theta_1\cap\{\x:g(\x)\geq  h(\x)\}$ and
 $\Theta_1^2=\Theta_1\cap\{\x:g(\x)< h(\x)\}$.
 Observe that $\x\in\Theta_1^1$ implies $h(\x)\leq 1$ and so $\Theta_1^1\subset \H_1$. Similarly
 $\x\in\Theta_1^2$ implies $g(\x)\leq 1$ and so $\Theta_1^2\subset \G_1$. Therefore
 ${\rm vol}(\Theta_1)\leq {\rm vol}(\G_1)+{\rm vol}(\H_1)<\infty$. And so $\theta\in\P[\x]_d$.
 \end{proof}

 With $y\in\R$ and $g\in\R[\x]$ let $\G_y:=\{\x: g(\x)\leq y\}$. 
 The following intermediate result which is crucial and of independent interest  was already proved in
 Morozov and Shakirov \cite{morosov1,morosov2} with different arguments.
   
 \begin{thm}
 \label{lemma1}
 Let $g\in\P[\x]_{d}$. Then for every $y\geq0$:
 \begin{equation}
 \label{lemma1-1}
 {\rm vol}(\G_y)\,=\,\frac{y^{n/d}}{\Gamma(1+n/d)}\,\int_{\R^n}\exp(-g(\x))\,d\x.
 \end{equation}
 \end{thm}
 \begin{proof}
 As $g\in\P[\x]_d$, and using homogeneity, ${\rm vol}(\G_1)<\infty$ implies ${\rm vol}(\G_y)<\infty$ for every $y\geq0$.
 Let $h:\R\to\R$ be the function $y\mapsto h(y):={\rm vol}(\G_y)$.
 Since $g$ is nonnegative, the function $h$ vanishes on $(-\infty,0]$. 
Its Laplace transform $\mathcal{L}[h]: \mathbb{C}\to\mathbb{C}$ is the function
\[\lambda\mapsto \mathcal{L}[h](\lambda):=\int_0^\infty\exp(-\lambda y) h(y)\,dy,\qquad\Re(\lambda)\,>\,0,\]
whose domain is $D=\{\lambda\in \C:\:\Re(\lambda)>0\}$.
Observe that whenever $\lambda\in\R$ with $\lambda>0$,
\begin{eqnarray*}
\mathcal{L}[h](\lambda)&=&\int_0^\infty\exp(-\lambda y)\left(\dis\int_{\{\x: g(\x)\leq y\}}d\x\right)\,dy\\
&=&\dis\int_{\R^n}\left( \int_{g(\x)}^\infty \exp(-\lambda y)dy\right)\,d\x\quad\mbox{[by Fubini's Theorem]}\\
&=&\frac{1}{\lambda}\dis\int_{\R^n}\,\exp(-\lambda g(\x))\,d\x\\
&=&
\frac{1}{\lambda}\dis\int_{\R^n}\,\exp(-g(\lambda^{1/d}\x))\,d\x\quad\mbox{[by homogeneity]}\\
&=&\frac{1}{\lambda^{1+n/d}}\dis\int_{\R^n}\exp(-g(\z))\,d\z\quad\mbox{[by $\lambda^{1/d}\x\to\z$]}\\
&=&\underbrace{\frac{\dis\int_{\R^n}\exp(-g(\z))\,d\z}{\Gamma(1+n/d)}}_{{\rm constant}\, c}\:\frac{\Gamma(1+n/d)}{\lambda^{1+n/d}}.
\end{eqnarray*}
Next, the function $\lambda\mapsto \frac{c\,\Gamma(1+n/d)}{\lambda^{1+n/d}}$
is analytic on $D$ and coincide with $\mathcal{L}[h]$ on the real half-line $\{t: t>0\}$ contained in $D$. Therefore
by the Identity Theorem 
$\mathcal{L}[h](\lambda)=\frac{c\,\Gamma(1+n/d)}{\lambda^{1+n/d}}$ on $D$. Finally observe that
$\frac{\Gamma(1+n/d)}{\lambda^{1+n/d}}$ is the Laplace transform of $t\mapsto u(t)=t^{n/d}$,
which yields the desired result $h(y)={\rm vol}(\G_y)=c\,y^{n/d}$. 
\end{proof}
And we also conclude:
\begin{corollary}
\label{coro1}
 Let $g\in\H[\x]_d$. Then $g\in\P[\x]_{d}$, i.e. ${\rm vol}(\G_1)<\infty$, if and only if 
 $\dis\int_{\R^n}\exp(-g(\x))d\x<\infty$.
\end{corollary}
\begin{proof}
The implication $\Rightarrow$ follows from Theorem \ref{lemma1}. For the reverse implication
consider the function $u:\R^n\times\R_+\to\R$ defined by $(\x,y)\mapsto u(\x,y):=\exp(-y)\,I_{\{\x: g(\x)\leq y\}}$,
which is measurable and nonnegative. Therefore by Tonelli's Theorem (see e.g. Royden \cite{royden}):
\begin{eqnarray*}
\int_{\R^n}\int_{\R_+} u(\x,y)\,d\x\,dy&=&
\int_{\R^n}\left(\underbrace{\int_{\R_+} u(\x,y)\,dy}_{\exp(-g(\x))}\right)\,d\x\,<\,\infty\\
&=&\int_{\R_+}\left(\int_{\R^n} u(\x,y)\,d\x\,\right)\,dy\\
&=&\int_{\R_+}\exp(-y)\,\left(\underbrace{\int_{\R^n} I_{\{\x: g(\x)\leq y\}}\,d\x}_{{\rm vol}\,(\G_y)}\right)\,dy.\end{eqnarray*}
Therefore ${\rm vol}\,(\G_y)$ is finite (and so is $\G_1$).
\end{proof}
As already mentioned, Formula (\ref{lemma1-1}) relating the Lebesgue volume $\G_1$ with the integral $\int\exp(-g)$ is already proved (with a different argument) in Morozov and Shakirov \cite{morosov1,morosov2}
where the authors want to express the non Gaussian integral $\int \exp(-g)$ in terms of algebraic invariants of $g$.

\subsection*{Sensitivity analysis and convexity}

We now investigate some properties of the function $v:\P[\x]_d\to\R$ defined by
\begin{equation}
\label{fonctionvolume}
g\,\mapsto\,v(g)\,:=\,{\rm vol}(\G_1)\,=\,\int_{\{\x\::\:g(\x)\leq1\}}d\x,\qquad g\in\P(\x]_{d},\end{equation}
i.e., we now view ${\rm vol}(\G_1)$ as a function of the 
vector $\g=(g_\alpha)\in\R^{\ell(d)}$ of coefficients of $g$ in the canonical basis of 
homogeneous polynomials of degree $d$ (and $\ell(d)={n+d-1\choose d}$).

\begin{thm}
\label{lemma2}
The {\it Lebesgue-volume} function $v:\P(\x]_d\to\R$ defined in (\ref{fonctionvolume}) is strictly convex
and lower semi-continuous.
In ${\rm int}(\P[\x]_d)$ its gradient $\nabla v$ and Hessian $\nabla^2v$ are given by:

\begin{equation}
\label{aux2-1}
\frac{\partial v(g)}{\partial g_\alpha}\,=\,\frac{-1}{\Gamma(1+n/d)}\int_{\R^n} \x^\alpha\,\exp(-g(\x))\,d\x,
\end{equation}
for all $\alpha\in\N^n_{d}$, $\vert\alpha\vert =d$.
\begin{equation}
\label{aux2-2}
\frac{\partial^2 v(g)}{\partial g_\alpha\partial g_\beta}\,=\,\frac{1}{\Gamma(1+n/d)}\int_{\R^n} \x^{\alpha+\beta}\,\exp(-g(\x))\,d\x,
\end{equation}
for all $\alpha,\beta\,\in\N^n_{d}$, $\vert\alpha\vert=\vert\beta\vert=d$.
Moreover, we also have
\begin{equation}
\label{euler-1}
\int_{\R^n}g(\x)\,\exp(-g(\x))\,d\x\,=\,
\frac{n}{d}\,\int_{\R^n}\exp(-g(\x))\,d\x.
\end{equation}
\end{thm}
\begin{proof}
By Theorem \ref{lemma1} $v(g)=c\int_{\R^n}\exp(-g(\x))d\x$ with $c=\Gamma(1+n/d)^{-1}$.
Let $p,q\in\P[\x]_{d}$ and $\alpha\in [0,1]$.
By convexity of $u\mapsto \exp(-u)$,
\begin{eqnarray*}
v(\alpha p+(1-\alpha)q)&\leq&c\,\dis\int_{\R^n} [\,\alpha\exp(-p(\x))+(1-\alpha)\exp(-q(\x))\,]\,d\x\\
&=&\alpha v(p)+(1-\alpha)v(q),
\end{eqnarray*}
and so $v$ is convex. Next, in view of the strict convexity of $u\mapsto \exp(-u)$, 
equality may occur only if $p(\x)=q(\x)$ almost everywhere, which implies $p=q$ and which in turn implies 
strict convexity of $v$. To get the lower-semicontinuity, let
$(g_n)\subset\P[\x]_d$ be  such that $g_n\to g$ as $n\to\infty$ and
$\liminf_{h\to g}v(h)=\lim_{n\to\infty}v(g_n)$. Then $g_n(\x)\to g(\x)$ pointwise and
by Fatou lemma (since $v\geq0$)
\begin{eqnarray*}
\liminf_{h\to g}v(h)\,=\,\liminf_{n\to\infty}\int_{\R^n}\exp(-g_n)\,d\x&\geq&\int_{\R^n}\liminf_{n\to\infty}\exp(-g_n)\,d\x\\
&=&\int_{\R^n}\exp(-g)\,d\x\,=\,v(g).\end{eqnarray*}

To obtain (\ref{aux2-1})-(\ref{aux2-2}) when $g\in {\rm int}(\P[\x]_d)$, one takes partial derivatives under the integral sign, which
in this context is allowed. Indeed,
write $g$ in the canonical basis as
 $g(\x)=\sum_{\vert\alpha\vert=d} g_\alpha\x^\alpha$. For every
$\alpha\in\N^n_d$ with $\vert\alpha\vert=d$, let $(e_\alpha)\subset\R^{\ell(d)}$ be the standard unit vectors
of $\R^{\ell(d)}$. Then for every $t>0$ sufficiently small, $\x\mapsto g(\x)+t\x^\alpha\in \P[\x]_d$ and
\[\frac{v(g+te_\alpha)-v(g)}{t}\,=\,c\,\int_{\R^n} \exp(-g)\,\left(\underbrace{\frac{\exp(-t\x^\alpha)-1}{t}}_{\psi(t,\x)}\right)\,d\x <\infty\]
Notice that for every $\x$, by convexity of the function $t\mapsto \exp(-t\x^\alpha)$, 
\[\lim_{t\downarrow0}\psi(t,\x)\,=\,\inf_{t>0}\psi(t,\x)\,=\,\exp(-t\x^\alpha)'_{\vert t=0}\,=\,-\x^\alpha,\]
because for every $\x$, the function $t\mapsto \psi(t,\x)$ is nondecreasing; see e.g. Rockafellar \cite[Theorem 23.1]{rockafellar}.
Hence, the one-sided directional derivative $v'(g;e_\alpha)$ in the direction $e_\alpha$ satisfies
\begin{eqnarray*}
v'(g;e_\alpha)&=&\lim_{t\downarrow 0}\frac{v(g+te_\alpha)-v(g)}{t}\,=\,
\lim_{t\downarrow 0}\,c\,\int_{\R^n}\exp(-g)\,\psi(t,\x)\,d\x\\
&=&c\,\int_{\R^n} \exp(-g)\,\lim_{t\downarrow 0}\psi(t,\x)\,d\mu(\x)\,=\,c\,\int_{\R^n}-\x^\alpha \exp(-g)\,d\x,
\end{eqnarray*}
where the third equality follows from the Extended Monotone Convergence Theorem \cite[1.6.7]{ash}.
Indeed for all $t<t_0$ with $t_0$ sufficiently small, the function $\psi(t,\cdot)$ is bounded above
by $\psi(t_0,\cdot)$ and $\int_{\R^n} \exp(-g)\psi(t_0,\x)d\mu<\infty$. Similarly, for every $t>0$
\[\frac{v(g-te_\alpha)-v(g)}{t}\,=\,
c\,\int_{\R^n}\exp(-g)\,\underbrace{\frac{\exp(t\x^\alpha)-1}{t}}_{\xi(t,\x)}\,d\x,\]
and by convexity of the function $t\mapsto \exp(t\x^\alpha)$
\[\lim_{t\downarrow0}\xi(t,\x)\,=\,\inf_{t >0}\xi(t,\x)\,=\,\exp(t\x^\alpha)'_{\vert t=0}\,=\,\x^\alpha.\]
Therefore, with exactly same arguments as before,
\begin{eqnarray*}
v'(g;-e_\alpha)&=&\lim_{t\downarrow 0}\frac{v(g-te_\alpha)-v(g)}{t}\\
&=&c\,\int_{\R^n}\x^\alpha\exp(-g)\,d\x=-v'(g;e_\alpha),\end{eqnarray*}
and so 
\[\frac{\partial v(g)}{\partial g_\alpha}\,=\,-c\,\int_{\R^n}\x^\alpha\,\exp(-g)\,d\x,\]
for every $\alpha$ with $\vert\alpha\vert= d$, which yields (\ref{aux2-1}). 
Similar arguments can used for the Hessian $\nabla^2v(g)$ which yields (\ref{aux2-2}).

To obtain (\ref{euler-1}) observe that $g\mapsto H(g):=\int \exp(-g)d\x$, $g\in\P[\x]_d$, is a positively
homogeneous function of degree $-n/d$, continuously differentiable. And so combining (\ref{aux2-1})
with Euler's identity $\langle \nabla H(g),g\rangle =-nH(g)/d$, yields:
\begin{eqnarray*}
-\frac{n}{d}\int_{\R^n}\exp(-g(\x))\,d\x&=&-\frac{n}{d}\,H(g)\\
&=&\langle \nabla H(g),g\rangle\quad\mbox{[by Euler's identity]}\\
&=&-\int_{\R^n}g(\x)\,\exp(-g(\x))\,d\x.\end{eqnarray*}
\end{proof}

Notice that convexity of $v$ is not obvious at all from its definition (\ref{fonctionvolume}) whereas it becomes almost transparent when using formula (\ref{lemma1-1}).
\subsection{The dual cone of $C_d(\K)$}
\label{dualck}

For a convex cone $C\subset\R^n$, the convex cone
\[C^*\,:=\,\{\,\y\::\: \langle \y,\x\rangle\,\geq\,0\quad\forall\,\x\,\in\,C\,\},\]
is called the {\it dual} cone of $C$, and if $C$ is closed then $(C^*)^*=C$.
Recall that for a set $\K\subset\R^n$, $C_d(\K)$ denotes the convex cone of polynomials of degree at most $d$ which are nonnegative on $\K$. We say that a vector $\y\in\R^{s(d)}$ has a representing measure (or is a $d$-truncated moment 
sequence) if there exists a finite Borel measure $\phi$ on $\R^n$ such that
\[y_\alpha\,=\,\int_{\R^n}\x^\alpha\,d\phi,\qquad \forall\alpha\in\N^n_d.\]

We will need the following (already known) characterization the dual cone $C_d(\K)^*$ (which is also
transparent in \cite[\S 1.1, p. 852]{nie-helton}).

\begin{lemma}
\label{dual-ck}
Let $\K\subset\R^n$ be compact. For every $d\in\N$, the dual cone $C_d(\K)^*$  is the convex cone
\begin{equation}\label{dual-ck-1}
\Delta_d\,:=\,\left\{\left(\int_\K \x^\alpha\,d\phi\right), \:\alpha\in \N^n_d\::\:\phi\in \mathcal{M}(\K)_+\:\right\},
\end{equation}
i.e., the convex cone of vectors of $\R^{s(d)}$ which have a representing measure with support contained in $\K$.
\end{lemma}
\begin{proof}
For every $\y=(y_\alpha)\in\Delta_d$ and $f\in C_d(\K)$ with coefficient vector $\f\in\R^{s(d)}$:
\begin{equation}
\label{dual-100}
\langle\y,\f\rangle\,=\,\sum_{\alpha\in\N^n_d}f_\alpha\, y_\alpha
\,=\,\sum_{\alpha\in\N^n_d}\int_\K f_\alpha\, \x^\alpha\,d\phi
\,=\,\int_\K f\,d\phi\,\geq\,0.\end{equation}
Since (\ref{dual-100}) holds for all $f\in C_d(\K)$ and all $\y\in\Delta_d$, then necessarily $\Delta_d\subseteq C_d(\K)^*$ and similarly,
$C_d(\K)\subseteq\Delta_d^*$. Next,
\begin{eqnarray*}
\Delta_d^*&=&\left\{\f\in\R^{s(d)}\::\:\langle\f,\y\rangle\geq0\quad\forall \y\in\Delta_d\right\}\\
&=&\left\{f\in\R[\x]_d\::\:\int_\K f\,d\phi\geq0\quad\forall \phi\in \mathcal{M}(\K)_+\right\}\\
&\Rightarrow&\Delta_d^*\subseteq C_{d}(\K),
\end{eqnarray*}
and so $\Delta_d^*=C_d(\K)$. Hence the result follows if one proves that $\Delta_d$ is closed, because then 
$C_d(\K)^*=(\Delta_d^*)^*=\Delta_d$, the desired result. So let $(\y^k)\subset\Delta_d$, $k\in\N$, with $\y^k\to\y$ as $k\to\infty$.
Equivalently, $\int_\K\x^\alpha d\phi_k\to y_\alpha$ for all $\alpha\in\N^n_d$.
In particular, the convergence $y^k_0\to y_0$ implies that the sequence of measures $(\phi_k)$, $k\in\N$, is bounded,
that is, $\sup_k\phi_k(\K)<M$ for some $M>0$. As $\K$ is compact, the unit ball of $\mathcal{M}(\K)$ is sequentially compact in the weak $\star$ topology $\sigma(\mathcal{M}(\K),C(\K))$ where $C(\K)$ is the space of continuous functions on $\K$. Hence there is a finite Borel measure 
$\phi\in \mathcal{M}(\K)_+$ and a subsequence $(k_i)$ such that $\int_\K gd\phi_{k_i}\to\int_\K gd\phi$ as $i\to\infty$, for all $g\in C(\K)$.
In particular, for every $\alpha\in\N^n_d$,
\[y_\alpha\,=\,\lim_{k\to\infty}y^{k}_\alpha\,=\,\lim_{i\to\infty}y^{k_i}_\alpha
\,=\,\lim_{i\to\infty}\int_\K\x^\alpha d\phi_{k_i}\,=\,\int_\K \x^\alpha d\phi,\]
which shows that $\y\in\Delta_d$, and so $\Delta_d$ is closed.
\end{proof}
And we also have:
\begin{lemma}
\label{empty}
Let $\K\subset\R^n$ be with nonempty interior. Then the interior of $C_d(\K)^*$ is nonempty.
\end{lemma}
\begin{proof}
Since $C_d(\K)$ is nonempty and closed, by Faraut and Kor\'any \cite[Prop. I.1.4, p. 3]{faraut}
\[{\rm int}(C_d(\K)^*)\,=\,\{\,\y\::\: \langle \y,\g\rangle >0,\quad\forall g\in C_d(\K)\setminus\{0\}\,\},\]
where $\g\in\R^{s(d)}$ is the coefficient of $g\in C_d(\K)$, and 
\[{\rm int}(C_d(\K)^*)\,\neq\,\emptyset\Longleftrightarrow C_d(\K)\cap (-C_d(\K))\,=\,\{0\}.\]
But $g\in C_d(\K)\cap (-C_d(\K))$ implies $g\geq0$ and $g\leq0$ on $\K$, which in turn implies $g=0$ 
because $\K$ has nonempty interior.
\end{proof}
For simplicity and with a slight abuse of notation, we will sometimes write 
$\langle\z,g\rangle$ in lieu of $\langle\z,\g\rangle$ and 
$\langle\z,1-g\rangle$ in lieu of $\langle\z,e_0-\g\rangle$ (where $e_0$ is the unit vector
corresponding to the constant polynomial equal to 1).

\section{Main result}

Consider the following problem $\P_0$, a non convex generalization of the L\"owner-John ellipsoid problem:\\

{\em $\P_0$: Let $\K\subset\R^n$ be a compact set not necessarily convex and $d$ an even integer. Find an homogeneous polynomial $g$ of degree $d$ such that its sublevel set
$\G_1:=\{\x\,:\,g(\x)\leq 1\}$ contains $\K$ and has minimum volume among all such sublevel sets
with this inclusion property.}

In the above problem $\P_0$, the set $\G_1$ is symmetric and so when $\K$ is a symmetric convex body and $d=2$, one retrieves the L\"owner-John ellipsoid problem in the symmetric case. In the next section we will consider the more general case where $\G_1$ is of the form $\G_1^\a:=\{\x:g(\x-\a)\leq 1\}$ for some $\a\in\R^n$ and some $g\in\P[\x]_d$.\\

Recall that $\P[\x]_{d}\subset\R[\x]_{d}$
is the convex cone of nonnegative homogeneous polynomials of degree $d$ whose
sublevel set $\G_1=\{\x:g(\x)\leq 1\}$ has finite volume. Recall also that $C_{d}(\K)\subset\R[\x]_{d}$ is the convex cone
of polynomials of degree at most $d$ that are nonnegative on $\K$.

We next show that solving $\P_0$ is equivalent to solving the convex optimization problem:

\begin{equation}
\label{minvolume}
\mathcal{P}:\qquad \rho=\dis\inf_{g\in\H[\x]_{d}}\:\left\{\,\int_{\R^n} \exp(-g)\,d\x\::\: 1-g\,\in\,C_{d}(\K)\,\right\}.
\end{equation}

\begin{proposition}
\label{min-vol-lemma}
Problem $\P_0$ has an optimal solution if and only if problem $\mathcal{P}$ in (\ref{minvolume}) has an optimal solution.
Moreover, $\mathcal{P}$ is a finite-dimensional convex optimization problem.
\end{proposition}
\begin{proof}
By Theorem \ref{lemma1}, 
\[\vol(\G_1)\,=\,\frac{1}{\Gamma(1+n/d)}\int_{\R^n}\exp(-g)\,d\x\]
whenever $\G_1$ has finite Lebesgue volume.
Moreover $\G_1$ contains $\K$ if and only if
$1-g\in C_{d}(\K)$ and so $\P_0$ has an optimal solution 
$g^*\in\P[\x]_{d}$ if and only if $g^*$ is an optimal solution of 
$\mathcal{P}$ (with value ${\rm vol}(\G^*_1)\Gamma(1+n/d)$).
Now since $g\mapsto \int_{\R^n}\exp(-g)d\x$ is strictly convex (by Lemma \ref{lemma2})
and both $C_{d}(\K)$ and $\P[\x]_d$ are convex cones, problem $\mathcal{P}$ is a finite-dimensional convex optimization problem.
\end{proof}
We now can state the first main result of this paper:
Recall that $\mathcal{M}(\K)_+$ is the convex cone of finite Borel measures on $\K$.

\begin{thm}
\label{vol-suff-cond}
Let $\K\subset\R^n$ be compact with nonempty interior and consider the convex optimization problem $\mathcal{P}$ in (\ref{minvolume}).

{\rm (a)} $\mathcal{P}$ has a unique optimal solution $g^*\in\P[\x]_{d}$.

{\rm (b)} Let $g^*\in\P[\x]_{d}$ be the unique optimal solution of $\mathcal{P}$ and let
$\G^*_1=\{\,\x:g^*(\x)\leq 1\,\}$. If $g^*\in{\rm int}(\P[\x]_d)$ then there exists a finite Borel measure $\mu^*\in \mathcal{M}(\K)_+$ 
such that
\begin{eqnarray}
\label{kkt-suff}
\int_{\R^n}\x^\alpha\exp(-g^*)d\x&=&\int_{\K}\x^\alpha\,d\mu^*,\qquad\forall\vert\alpha\vert=d\\
\label{kkt-suff1}
\int_\K(1-g^*)\,d\mu^*&=&0;\quad \mu^*(\K)=\frac{n}{d}\int_{\R^n}\exp(-g^*)\,d\x.
\end{eqnarray}
In particular, $\mu^*$ is supported on the set $V:=\{\x\in\K: g^*(\x)=1\}\,(=\K\cap\G^*_1)$ and in fact,
$\mu^*$ can be substituted with another 
measure $\nu^*\in \mathcal{M}(\K)_+$ supported on at most ${n+d-1\choose d}$ contact points of $V$.

{\rm (c)} Conversely, if $g^*\in\R[\x]_{d}$ is homogeneous with $1-g^*\in C_d(\K)$, and 
there exist points $(\x_i,\lambda_i)\in \K\times\R$, $\lambda_i>0$, $i=1,\ldots,s$, such that $g^*(\x_i)=1$ for all $i=1,\ldots,s$, and 
\[\int_{\R^n}\x^\alpha\exp(-g^*)\,d\x\,=\,\sum_{i=1}^s\lambda_i\,\x_i^\alpha,\qquad \vert\alpha\vert=d,\]
then $g^*$ is the unique optimal solution of problem $\mathcal{P}$.
\end{thm}

The proof is postponed to \S \ref{proofs}. 
Importantly, notice that neither $\K$ nor $\G^*_1$ are required to be convex.
If the optimal solution $g^*\not\in{\rm int}(\P[\x]_d)$ then $\mu^*$ satisfies an analogue of (\ref{kkt-suff}) 
which now involves a subgradient $\partial v(g^*)$ at $g^*$ of the function $g\mapsto v(g)=\int \exp(-g)d\x$. That is,
$(\int_\K\x^\alpha d\mu^*)_{\vert\alpha\vert=d}\in-\partial v(g^*)$.

\subsection{On the contact points}
Theorem \ref{vol-suff-cond} states that $\mathcal{P}$ (hence $\P_0$) has a unique optimal solution $g^*\in\P[\x]_d$ and 
if $g^*\in{\rm int}(\P[\x]_d)$ one may find contact points $\x_i\in\K\cap\G^*_1$, $i=1,\ldots,s$, with $s\leq 
{n+d-1\choose d}$, such that
\begin{equation}
\label{aux-100}
y^*_\alpha\,=\,\int_{\R^n}\x^\alpha\,\exp(-g^*(\x))\,d\x\,=\,\sum_{i=1}^s\lambda_i\,\x_i^\alpha,\quad\vert\alpha\vert=d,
\end{equation}
for some positive weights $\lambda_i$. In particular, using the identity
(\ref{euler-1}) and $\langle 1-g^*,\y^*\rangle=0$, as well as $g^*(\x_i)=1$ for all $i$,
\[y^*_0\,=\,\sum_{\vert\alpha\vert=d}y^*_\alpha\,g^*_\alpha\,=\,
\frac{n}{d}\int_{\R^n} \exp(-g^*(\x))\,d\x\,=\,\sum_{i=1}^s \lambda_i.\]
Next, recall that $d$ is even and let $\v_{d/2}:\R^n\to\R^{\ell(d/2)}$ be the mapping
\[\x\mapsto \v_{d/2}(\x)=(\x^\alpha),\qquad\vert\alpha\vert\,=\,d/2,\]
i.e., the ${n-1+d/2\choose d/2}$-vector of the canonical basis of $\H[\x]_{d/2}$. From (\ref{aux-100}),
\begin{equation}
\label{second}
\int_{\R^n}\v_{d/2}(\x)\,\v_{d/2}(\x)^T\exp(-g^*)\,d\x\,=\,\sum_{i=1}^s \lambda_i\,\v_{d/2}(\x_i)\,\v_{d/2}(\x_i)^T.\end{equation}
Hence, when $d=2$ and $\K$ is symmetric, one retrieves the characterization in John's theorem
\cite[Theorem 2.1]{henk},
namely that if the euclidean ball $\xi_n:=\{\x:\Vert\x\Vert\leq 1\}$ is the unique ellipsoid of
minimum volume containing $\K$ then there are contact points $(\x_i)\subset\xi_n\cap\K$
and positive weights $(\lambda_i)$, such that $\sum_i\lambda_i\x_i\x_i^T=I_n$ (where $I_n$ is the 
$n\times n$ identity matrix). Indeed in this case, $\v_{d/2}(\x)=\x$, $g^*(\x)=\Vert\x\Vert^2$ and 
$\int_{\R^n}\x\x^T\exp(-\Vert\x\Vert^2)d\x=c\,I_n$ for some constant $c$.

So (\ref{second}) is the analogue for $d>2$ of the contact-points property in John's theorem and we obtain the following generalization: For $d$ even, let $\Vert\x\Vert_d:=(\sum_{i=1}^nx_i^d)^{1/d}$ denote the 
$d$-norm with unit ball $\xi^d_n:=\{\x:\Vert \x\Vert_d\leq 1\}$.
\begin{corollary}
If in Theorem \ref{vol-suff-cond} the unique optimal solution $\G^*_1$ is the $d$-unit ball $\xi^d_n$ then there are contact points $(\x_i)\subset\K\cap\xi^d_n$ and positive weights $\lambda_i$, $i=1,\ldots,s$, with $s\leq {n+d-1\choose d}$, such that for every $\vert\alpha\vert=d$,
\[\sum_{i=1}^s\lambda_i\,\v_{d/2}(\x_i)\v_{d/2}(\x_i)^T\,=\,\int_{\R^n}\v_{d/2}(\x)\v_{d/2}(\x)^T\exp(-\Vert\x\Vert^d_d)\,d\x.\]
\end{corollary}
Equivalently, for $\vert\alpha\vert=d$,
\[\sum_{i=1}^s\lambda_i\,\x_i^\alpha\,=\,\left\{\begin{array}{ll}
\dis\prod_{j=1}^n\dis\int_\R t^{\alpha_j}\exp(-t^d)\,dt&\mbox{if $\alpha=2\beta$}\\
0&\mbox{otherwise.}\end{array}\right.\]

\section*{Example}
With $n=2$ let $\K\subset\R^2$ be the box $[-1,1]^2$ and let $d=4,6$, that is, one searches for 
the unique homogeneous polynomial $g\in\R[\x]_4$ or $g\in\R[\x]_6$ which contains $\K$ and has minimum volume among such sets. 
\begin{thm}
\label{th4}
The sublevel set $\G^4_1=\{\,\x:g_4(\x)\leq1\,\}$ associated with the homogeneous polynomial
\begin{equation}
\label{th4-1}
\x\mapsto g_4(\x)\,=\,x_1^4+x_2^4-x_1^2x_2^2,\end{equation}
is the unique solution of problem $\P_0$ with $d=4$. That is, $\K\subset\G_1^4$ and $\G^4_1$ has minimum volume among all sets
$\G_1\supset\K$ defined with homogeneous polynomials of degree $4$.

Similarly, the sublevel set $\G^6_1=\{\,\x:g_6(\x)\leq1\,\}$ associated with the homogeneous polynomial
\begin{equation}
\label{th4-2}
\x\mapsto g_6(\x)\,=\,x_1^6+x_2^6-(x_1^4x_2^2+x_1^2x_2^4)/2
\end{equation}
is the unique solution of problem $\P_0$ with $d=6$.
\end{thm}
\begin{proof}
Let $g_4$ be as in (\ref{th4-1}) (hence in ${\rm int}(\P[\x]_d)$). We first prove that  $\K\subset\G_1^4$, i.e., 
$1-g_4(\x)\geq0$ whenever $\x\in\K$. But observe that if $\x\in\K$ then
\begin{eqnarray*}
1-g_4(\x)&=&1-x_1^4-x_2^4+x_1^2x_2^2\,=\,1-x_1^4+x_2^2(x_1^2-x_2^2)\\
&\geq&1-x_1^4+x_2^2(x_1^2-1)\quad\mbox{[as $-x_2^2\geq-1$ and $x_2^2\geq0$]}\\
&\geq&(1-x_1^2)\,(1+x_1^2-x_2^2)\\
&\geq&(1-x_1^2)\,x_1^2\,\geq\,0\quad\mbox{[as $-x_2^2\geq-1$ and $1-x_1^2\geq0$]}.
\end{eqnarray*}
Hence $1-g_4\in C_d(\K)$. Observe that $\K\cap\G_1^4$ consists of the $8$ contact points
$(\pm 1,\pm1)$ and $(0,\pm 1)$, $(\pm 1,0)$.
Next let $\nu^*$ be the measure 
defined by 
\begin{equation}
\label{nu*}
\nu^*\,=\,a\,(\delta_{(-1,1)}+\delta_{(1,1)})+b\,(\delta_{(1,0)}+\delta_{(0,1)})
\end{equation}
where $\delta_\x$ denote the Dirac measure at $\x$ and $a,b\geq0$ are chosen to satisfy
\[2\,a\,+\,b\,=\,\int_{\R^n}x_1^4\,\exp(-g_4)\,d\x;\quad 2\,a\,=\,\int_{\R^n}x_1^2x_2^2\,\exp(-g_4)\,d\x,\]
so that 
\[\int\x^\alpha\,d\nu^*\,=\,\int_{\R^n}\x^\alpha\,\exp(-g_4(\x))\,d\x,\quad\vert\alpha\vert=4.\]
Of course a unique solution $(a,b)\geq0$ exists since
\begin{eqnarray*}
\int_{\R^n}x_1^4\,\exp(-g_4)\,d\x\int_{\R^n}x_2^4\,\exp(-g_4)\,d\x&=&\left(\int_{\R^n}x_1^4\,\exp(-g_4)\,d\x\right)^2\\
&\geq&\left(\int_{\R^n}x_1^2x_2^2\,\exp(-g_4)\,d\x\right)^2.\end{eqnarray*}
Therefore the measure $\nu^*$ is indeed as in Theorem \ref{vol-suff-cond}(c) and the proof is completed.
Notice that  as predicted by Theorem \ref{vol-suff-cond}(b), $\nu^*$ is supported on
$4\leq {n+d-1\choose d}=5$  points.
Similarly with $g_6$ as in (\ref{th4-2}) and $\x\in\K$,
\begin{eqnarray*}
1-g_6(\x)&=&1-x_1^6-x_2^6+(x_1^4x_2^2+x_1^2x_2^4)/2\\
&=&(1-x_1^6)/2+(1-x_2^6)/2-x_1^4(x_1^2-x_2^2)/2-x_2^4(x_2^2-x_1^2)/2\\
&\geq&(1-x_1^6)/2+(1-x_2^6)/2-x_1^4(1-x_2^2)/2-x_2^4(1-x_1^2)/2\\
&&\mbox{[as $-x_1^6\geq-x_1^4$ and $-x_2^6\geq-x_2^4$]}\\
&\geq&(1-x_1^2)(1+x_1^2+x_1^4-x_2^4)/2+(1-x_2^2)(1+x_2^2+x_2^4-x_1^4)/2\\
&\geq&(1-x_1^2)(x_1^2+x_1^4)/2+(1-x_2^2)(x_2^2+x_2^4)/2\,\geq0\\
&&\mbox{[as $1-x_1^4\geq0$ and $1-x_2^4\geq0$]}.
\end{eqnarray*}
So again the measure $\nu^*$ defined in (\ref{nu*}) where
$a,b\geq0$ are chosen to satisfy
\[2\,a\,+\,b\,=\,\int_{\R^n}x_1^6\,\exp(-g_6)\,d\x;\quad 2\,a\,=\,\int_{\R^n}x_1^4x_2^2\,\exp(-g_6)\,d\x,\]
is such that
\[\int\x^\alpha\,d\mu^*\,=\,\int_{\R^n}\x^\alpha\,\exp(-g_6(\x))\,d\x,\quad\vert\alpha\vert=6.\]
Again a unique solution $(a,b)\geq0$ exists because
\[\left(\int_{\R^n}x_1^6\,\exp(-g_6)\,d\x\right)\left(\int_{\R^n}x_1^2x_2^4\,\exp(-g_6)\,d\x\right)
\,\geq\,\left(\int_{\R^n}x_1^4x_2^2\,\exp(-g_6)\,d\x\right)^2.\]
\end{proof}
With $d=4$, the non convex sublevel set $\G_1^4=\{\x:g_4(\x)\leq1\}$ which is displayed in Figure \ref{figure-degree4} (left) is a much better 
approximation of $\K=[-1,1]^2$ than the ellipsoid of minimum volume $\xi=\{\x:\Vert\x\Vert^2\leq 2\}$ that contains 
$\K$. In particular, ${\rm vol}(\xi)=2\pi\approx 6.28$ whereas ${\rm vol}(\G_1^4)\approx 4.32$. 
\begin{figure}
\resizebox{1.1\textwidth}{!}
{\includegraphics{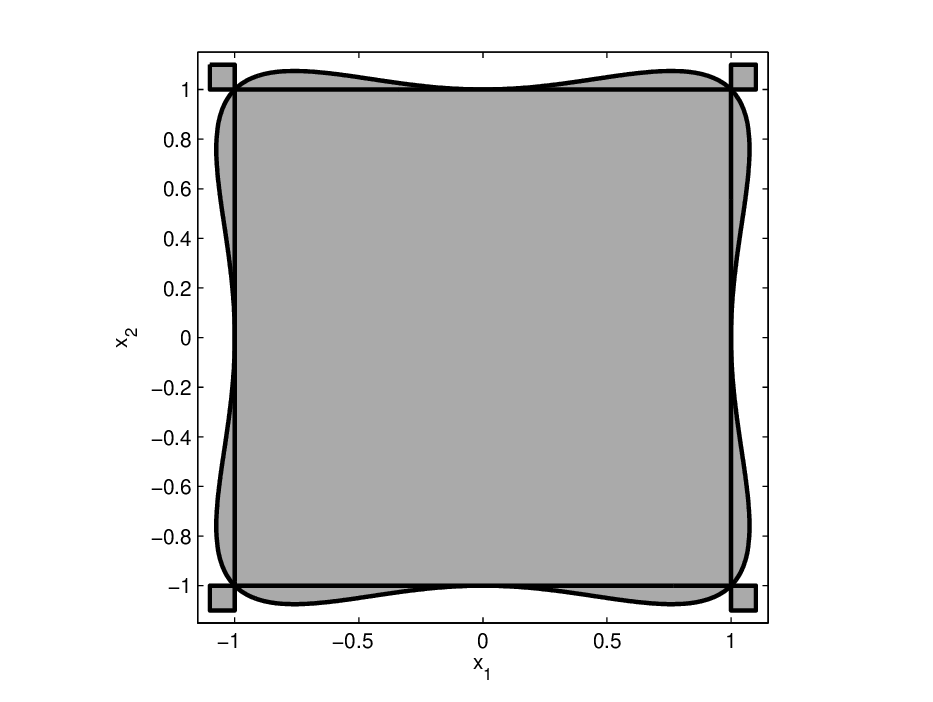}\includegraphics{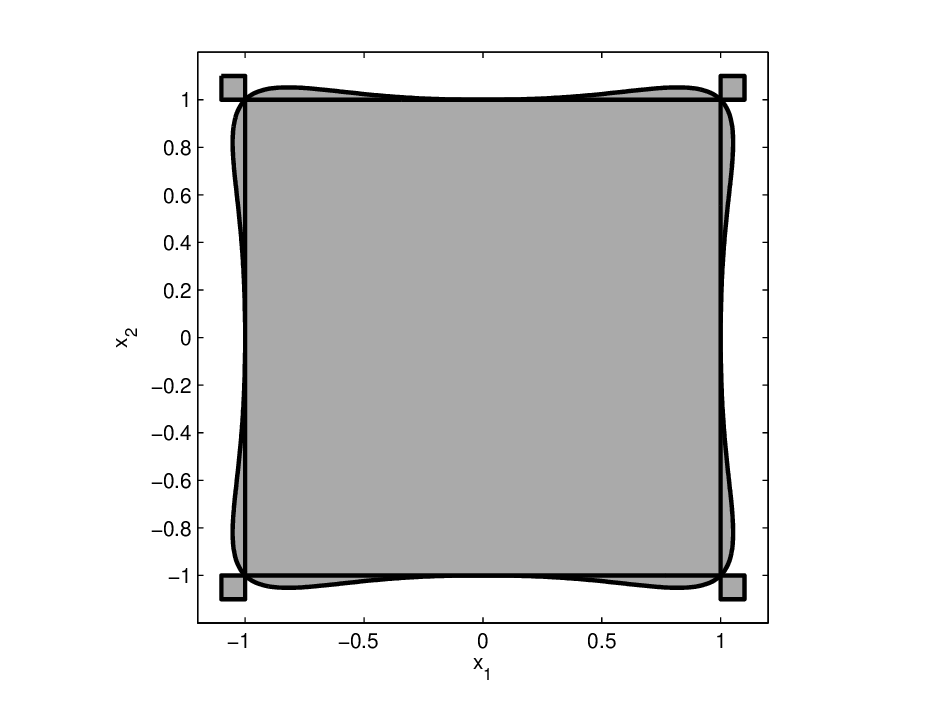}}
\caption{$\K=[-1,1]^2$ and $\G_1^4=\{\x:x_1^4+x_2^4-x_1^2x_2^2\leq1\}$ (left),
$\G_1^6=\{\x:x_1^6+x_2^6-(x_1^4x_2^2+x_1^2x_2^4)/2\leq1\}$ (right)\label{figure-degree4}}
\end{figure}
With $d=6$, the non convex sublevel set $\G_1^6=\{\x:g_6(\x)\leq1\}$ which is displayed in Figure \ref{figure-degree4} (right) is again a better 
approximation of $\K=[-1,1]^2$ than the ellipsoid of minimum volume $\xi=\{\x:\Vert\x\Vert^2\leq 2\}$ that contains 
$\K$, and as ${\rm vol}(\G_1^6)\approx 4.1979$ it provides a better approximation than the sublevel set $\G_1^4$ with $d=4$.

Finally if $\K=\G_1^6$ then $\G_1^4$ is an optimal solution of $\P_0$ with $d=4$, 
that is $\G_1^4$ has minimum volume among all sets $\G_1\supset\G_1^6$ defined by homogeneous polynomials $g\in\P[\x]_4$.
Indeed first we have solved the polynomial optimization problem:
$\rho=\inf_\x\{1-g_4(\x): 1-g_6(\x)\geq0\}$
via the hierarchy of semidefinite relaxations\footnote{We have used the GloptiPoly software \cite{gloptipoly}
dedicated to solving the generalized problem of moments.} defined in \cite{lasserresiopt,lasserrebook2} and at the fifth semidefinite relaxation (i.e. with moments of order $12$) we found
$\rho=0$ with the eight contact points $(\pm 1,\pm1),(\pm 1,0),(0,\pm 1)\in\G_1^4\cap\G_1^6$ as global minimizers!
This shows (up to $10^{-9}$ numerical precision) that $\G_1^6\subset\G_1^4$.
Then again the measure $\nu^*$ defined in (\ref{nu*}) satisfies Theorem \ref{vol-suff-cond}(b)
and so $g_4$ is an optimal solution of problem $\P_0$ with $\K=\G_1^6$ and $d=4$.

At last, the ball $\G_1^2=\{\x: (x_1^2+x_2^2)/2\leq1\}$ is an optimal solution of $\P_0$ with $d=2$
and we have $\K\subset\G_1^6\subset\G_1^4\subset\G_1^2$.

\section{The general case}

We now consider the more general case where the set $\G_1$ is of the form
$\{\x:g(\x-\a)\leq 1\}=:\G_1^\a$ where $\a\in\R^n$ and $g\in\P[\x]_d$.

For every $\a\in\R^n$ and $g\in\R[\x]_d$ (with coefficient vector $\g\in\R^{s(d)}$) define the polynomial $g_\a\in\R[\x]_d$ by
$\x\mapsto g_\a(\x):=g(\x-\a)$ and its sublevel set $\G_1^\a:=\{\x:g_\a(\x)\leq1\}$. 
The polynomial $g_\a$ can be written
\begin{equation}
\label{def}
g_\a(\x)\,=\,g(\x-\a)\,=\,\sum_{\alpha\in\N^n_d}p_\alpha(\a,\g)\,\x^\alpha,\end{equation}
where $\g\in\R^{s(d)}$ and the polynomial $p_\alpha\in\R[\x,\g]$ is linear in $\g$, for every
$\alpha\in\N^n_d$.
Consider the following generalization of $\P_0$:

{\em $\P$: Let $\K\subset\R^n$ be a compact set not necessarily convex and $d\in\N$ an even integer. Find an homogeneous polynomial $g$ of degree $d$ and a point $\a\in\R^n$ such that the sublevel set
$\G_1^\a:=\{\x\,:\,g(\x-\a)\leq 1\}$ contains $\K$ and has minimum volume among all such sublevel sets
with this inclusion property.}

When $d=2$ one retrieves the general (non symmetric) L\"owner-John ellipsoid problem. For $d>2$,
an even more general problem would be to find a (non homogeneous) polynomial $g$
of degree $d$ such that $\K\subset\G_1=\{\x:g(\x)\leq 1\}$ and $\G_1$ has minimum volume among
all such set $\G_1$ with this inclusion property. However when $g$ is not homogeneous
we do not have an analogue of Theorem \ref{lemma1} for the Lebesgue-volume ${\rm vol}(\G_1)$.\\

So in view of (\ref{newa}), one wishes to solve the optimization problem
\begin{equation}
\label{minvolume-a}
\mathcal{P}:\quad\rho=\min_{\a\in\R^n,g\in\P[\x]_d}\:\{{\rm vol}(\G^\a_1)\::\:1-g_\a\in C_d(\K)\},\end{equation}
a generalization of (\ref{minvolume}) where $\a=0$. In contrast to $\P_0$, problem $\P$ is not convex and so
computing a global optimal solution is more difficult. In particular we do not provide
an analogue of the numerical scheme for $\P_0$ described in \S \ref{numerical} and the results of this section are mostly of theoretical interest. However we still can show that for every
optimal solution $(\a^*,g^*)\in\R^n\times\P[\x]_d$, there is a characterization of
$\G_1^{\a^*}$ similar to the one obtained for $\P_0$.

Let $\K-\a$
denotes the set $\{\x-\a:\x\in\K\}$, and observe that whenever $g\in\P[\x]_d$,
\begin{equation}
\label{newa}
{\rm vol}(\G_1^\a)\,=\,{\rm vol}(\G_1^0)\,=\,{\rm vol}(\G_1)\,=\,\frac{1}{\Gamma(1+n/d)}\int_{\R^n}\exp(-g(\x))\,d\x.\end{equation}

\begin{thm}
\label{vol-suff-cond-a}
Let $\K\subset\R^n$ be compact with nonempty interior and consider the optimization problem $\mathcal{P}$ in (\ref{minvolume-a}).

{\rm (a)} $\mathcal{P}$ has an optimal solution $(\a^*,g^*)\in\R^n\times \P[\x]_{d}$.

{\rm (b)} Let $(\a^*,g^*)\in\R^n\times \P[\x]_{d}$ be an optimal solution of $\mathcal{P}$. 
If $g^*\in{\rm int}(\P[\x]_d)$ then there exists a finite Borel measure $\mu^*\in \mathcal{M}(\K-\a^*)_+$ 
such that
\begin{eqnarray}
\label{kkt-suff-a}
\int_{\R^n}\x^\alpha\exp(-g^*)d\x&=&\int_{\K-\a^*}\x^\alpha\,d\mu^*,\qquad\forall\vert\alpha\vert=d\\
\label{kkt-suff1-a}
\int_{\K-\a^*}(1-g^*)\,d\mu^*&=&0;\quad \mu^*(\K-\a^*)=\frac{n}{d}\int_{\R^n}\exp(-g^*)\,d\x.
\end{eqnarray}
In particular, $\mu^*$ is supported on the set $V:=\{\x\in\K-\a^*: g^*(\x)=1\}\,(=\K\cap\G^{\a^*}_1)$ and in fact,
$\mu^*$ can be substituted with another 
measure $\nu^*\in \mathcal{M}(\K-\a^*)_+$ supported on at most ${n+d-1\choose d}$ contact points of $V$ with same moments
of order $d$.
\end{thm}

The proof is postponed to \S \ref{conda}

\section{A computational procedure}
\label{numerical}
Even though $\mathcal{P}$ in (\ref{minvolume}) is a finite-dimensional convex optimization problem,
it  is hard to solve for mainly two reasons:
\begin{itemize}
\item From Theorem \ref{lemma2}, the gradient and Hessian of the (strictly) convex objective function $g\mapsto \int\exp(-g)$ requires evaluating  integrals of the form
\[\int_{\R^n}\x^\alpha\exp(-g(\x))\,d\x, \quad\forall\alpha\in\N^n_d,\]
a difficult and challenging problem. (And with $\alpha=0$ one obtains the value of the objective function.)
\item The convex cone $C_d(\K)$ has no {\em exact} and  {\em tractable} representation to efficiently handle the constraint $1-g\in C_d(\K)$ in an algorithm for solving problem (\ref{minvolume}).
\end{itemize}
However, below we outline a numerical scheme to 
approximate to any desired $\epsilon$-accuracy (with $\epsilon>0$):

- the optimal value $\rho$ of (\ref{minvolume}),

- the unique optimal solution $g^*\in\P[\x]_d$ of $\mathcal{P}$ obtained in Theorem \ref{vol-suff-cond}.

\subsection{Concerning gradient and Hessian evaluation}
\label{gradient-hessian}
To approximate the gradient and Hessian of the objective function we will use the following result:
\begin{lemma}
Let $g\in{\rm int}(\P[\x]_d)$ with $\G_1=\{\x:g(\x)\leq 1\}$. Then for all $\alpha\in\N^n$
\begin{equation}
\label{gradient-eval}
\int_{\R^n}\x^{\alpha}\exp(-g)\,d\x\,=\,\Gamma\left(1+\frac{n+\vert\alpha\vert}{d}\right)\int_{\G_1}\x^\alpha\,d\x.
\end{equation}
\end{lemma}
The proof being identical to that of Theorem \ref{lemma1} is omitted. So Lemma \ref{gradient-eval}
relates in a very simple and explicit manner
all moments of the Borel measure with density $\exp(-g)$ on $\R^n$ 
with those of the Lebesgue measure on the sublevel set $\G_1$.

It turns out that in Henrion et al. \cite{sirev} we have provided a hierarchy of semidefinite programs\footnote{A semidefinite program is a finite-dimensional convex optimization problem which in canonical form reads:
$\min_\x\{\c^T\x:\A_0+\sum_{k=1}^t\A_kx_k\succeq0\}$, where $\c\in\R^t$ and the $\A_k$'s are real symmetric matrices.
Importantly, up to arbitrary fixed precision it can be solved in time polynomial in the input size of the problem.}
to approximate
as closely as desired, any finite moment sequence $(z_\alpha)$, $\alpha\in\N^n_\ell$, defined by 
\[z_\alpha\,=\,\int_\Omega\x^\alpha\,d\x,\quad\alpha\in\N^n_\ell.\]
where $\Omega$ is a compact basic semi-algebraic set of the form
$\{\x:g_j(\x)\geq0,\:j=1,\ldots,m\}$ for some polynomials $(g_j)\subset\R[\x]$. 
Let us briefly explain how it works when $\Omega=\G_1$ and $\G_1$ is bounded. Let $\B\supset\G_1$ be a box that contains $\G_1$ and let $\lambda$ be the restriction
of the Lebesgue measure on $\B$ of which moments
\[\lambda_\alpha\,=\,\int_\B \x^\alpha\,d\x,\qquad\alpha\in\N^n,\]
are easy to compute. Write $\B$ as $\{\,\x:\theta_i(\x)\geq0,\,i=1,\ldots,n\,\}$ where 
$\x\mapsto \theta_i(\x)=(\overline{a_i}-x_i)(x_i-\underline{a_i})$ for some 
scalars $(\overline{a_i},\underline{a_i})$ that define the box $\B$.
Then ${\rm vol}(\G_1)$ is the optimal value of the optimization problem:
\begin{equation}
\label{volsirev}
\sup_{\mu,\nu}\,\{\,\mu(\R^n):\:\mu+\nu=\lambda;\:\mu(\B\setminus\G_1)=0,\quad \mu,\nu\in \mathcal{M}(\B)_+\,\},\end{equation}
where $\mathcal{M}(\B)_+$ is the space of finite Borel measures on $\B$. The dual of the above problem reads:
\begin{equation}
\label{volsirevdual}
\inf_{p\in\R[\x]}\:\{\,\int_\B p(\x)\,\lambda(d\x)\,:\: p(\x)\geq0\mbox{ on $\B$};\:p(\x)\geq1\mbox{ on $\G_1$}\,\}.\end{equation}
In the dual a minimizing sequence $(p_k)\subset\R[\x]$, $k\in\N$, approximates the indicator function 
$1_{\G_1}$  of the set $\G_1$ by polynomials nonnegative on $\B$ and
of increasing degree. 
To approximate ${\rm vol}(\G_1)$
we proceed as in \cite{sirev} and use the following hierarchy of semidefinite relaxations of (\ref{volsirev}) indexed by $k\in\N$:
\begin{equation}
\label{sdp-1}
\begin{array}{rll}
\rho_k=\displaystyle\sup_{\y,\z} &y_0&\\
\mbox{s.t.}&y_\alpha+z_\alpha&=\lambda_\alpha,\quad\alpha\in\N^n;\quad \vert\alpha\vert\leq 2k\\
&\M_k(\y),\M_k(\z)&\succeq0\\
&\M_{k-d/2}(1-g\,\y)&\succeq0\\
&\M_{k-1}(\theta_i\,\z)&\succeq0,\quad i=1,\ldots,n,
\end{array}\end{equation}
where $\y=(y_\alpha)$ (resp. $\z=(z_\alpha)$), $\alpha\in\N^n$, is a sequence that approximates
the moment sequences of $\mu$ (resp. $\nu$). The matrix $\M_k(\y)$ is the 
moment matrix of order $k$ associated with $\y$, whereas $\M_{k-d/2}(1-g\,\y)$ (resp. $\M_{k-1}(\theta_i\,\y)$) is the localizing matrix 
associated with $\y$ and $1-g$ (resp. $\theta_i$); see e.g. \cite{sirev}.

For each $k\in\N$, (\ref{sdp-1}) is a semidefinite program
and in \cite{sirev} it is proved that
$(\rho_k)$ is monotone nonincreasing and 
$\rho_k\to{\rm vol}(\G_1)$ as $k\to\infty$. In addition, let $\y^k=(\y^k_\alpha)$, $\alpha\in\N^n_{2k}$,
be an optimal solution of (\ref{sdp-1}). Then for each fixed $\alpha\in\N^n$,
\[y^k_\alpha\to\,\int_{\G_1}\x^\alpha\,d\x,\qquad\mbox{as $k\to\infty$.}\]
For more details, the interested reader is referred to \cite{sirev}.
Not surprisingly, it is hard to approximate $1_{\G_1}$
by polynomials and in particular this is reflected by the well-known Gibbs effect
in the dual (\ref{volsirevdual}) (and hence in the dual of (\ref{sdp-1})), which can make the convergence 
$\rho_k\to{\rm vol}(\G_1)$ slow. Below we show how one can drastically improve this convergence and fight the Gibbs effect.

\subsection*{Improving the above algorithm.}

Observe that in (\ref{sdp-1}) we have not used the fact that $g$ is 
homogeneous of degree $d$. However from Lemma 1 in Lasserre \cite{lasserre-dcg}, for every $k\in\N$ one has:
\begin{equation}
\label{dcg-lem-1}
\int_{\{\x\,:\,g(\x)\leq 1\}}\x^\alpha\,g(\x)^k\,d\x\,=\,\frac{n+\vert\alpha\vert}{n+kd+\vert\alpha\vert}\,\int_{\{\x\,:\,g(\x)\leq 1\}}\,\x^\alpha\,d\x.\\
\end{equation}
Therefore if we write
$g(\x)=\sum_\beta g_\beta \x^\beta$, then (\ref{dcg-lem-1}) with $k=1$ translates into the linear equality constraints
\begin{equation}
\label{help-1}
\sum_{\vert\beta\vert=d}g_\beta\, y_{\alpha+\beta}\,=\,\frac{n+\vert\alpha\vert}{n+d+\vert\alpha\vert}\,y_\alpha,
\qquad \alpha\in\N^n,
\end{equation}
on the moments $(y_\alpha)$ of the Lebesgue measure on $\G_1$. So we may and will
include the linear constraints (\ref{help-1}) in the semidefinite program (\ref{sdp-1}), which yields
the resulting semidefinite program:
\begin{equation}
\label{sdp-2}
\begin{array}{rll}
\tilde{\rho_k}=\displaystyle\sup_{\y,\z} &y_0&\\
\mbox{s.t.}&y_\alpha+z_\alpha&=\lambda_\alpha,\quad\alpha\in\N^n;\quad \vert\alpha\vert\leq 2k\\
&\M_k(\y),\M_k(\z)&\succeq0\\
&\M_{k-d/2}(1-g\,\y)&\succeq0\\
&\M_{k-1}(\theta_i\,\z)&\succeq0,\quad i=1,\ldots,n\\
&\displaystyle\sum_{\vert\beta\vert=d}g_\beta\, y_{\alpha+\beta}&=\frac{n+\vert\alpha\vert}{n+d+\vert\alpha\vert}y_\alpha,\quad\vert\alpha\vert\leq 
2k-d,
\end{array}\end{equation}
and obviously ${\rm vol}(\G_1)\leq\tilde{\rho}_k\leq\rho_k$ for all $k$. 
To appreciate how powerful can be these additional constraints,
consider the following two simple illustrative examples:
\begin{example}
\label{ex11}
{\rm Let $n=1$ and let $\G_1:=\{x: 4x^2\leq 1\}\subset\B=[-1,1]$ so that ${\rm vol}(\G_1)=1$.
Table \ref{tab1} below displays results obtained by solving
(\ref{sdp-1}) and (\ref{sdp-2}) respectively. 
As one may see in Table \ref{tab1}, the convergence 
$\rho_k\to 1$ is rather slow (because of the Gibbs effect in the dual) whereas
the convergence $\tilde{\rho_k}\to 1$ is very fast.
And indeed $\tilde{\rho_k}$ provides with a much better approximation of ${\rm vol}(\G_1)$
than $\rho_k$; already with moments up to order $10$ only, $\tilde{\rho_5}$ 
provides with a very good approximation.

\begin{table}
\caption{Comparing $\rho_k$ and $\tilde{\rho_k}$ for the interval $[-1/2,1/2]$.\label{tab1}}
{\begin{tabular}{@{}ccc@{}}
$k$& $\rho_k$ & $\tilde{\rho_k}$\\
4&1.689&1.156\\
6&1.463&1.069\\
8&1.423&1.025\\
10&1.382&1.010\\
12&1.305&1.003\\
14&1.289&1.001\\
16&1.267&1.000\\
18&1.229&1.000\\
20&1.221&1.000\\
\end{tabular}}
\end{table}
}\end{example}

\begin{example}
\label{ex22}
{\rm Let $n=2$ and let $\G_1=\{\x: \Vert\x\Vert^2\leq 1\}$ be the unit ball with volume $\pi$.
Table \ref{tab1} below displays results obtained by solving
(\ref{sdp-1}) and (\ref{sdp-2}) respectively. Of course the precision also depends on the size of the box 
$\B$ that contains $\G_1$. And so we have taken a box $\B=[-a,a]^2$ with $a$ ranging from $1$ to $2$.
As one may see in Table \ref{tab2}, $\tilde{\rho_k}$ is a much better approximation of $\pi$
than $\rho_k$ and already with moments up to order $8$ only, quite good approximations
are obtained.

\begin{table}
\caption{Comparing $\rho_k$ and $\tilde{\rho_k}$ for the unit sphere.\label{tab2}}
{\begin{tabular}{@{}ccccc@{}}
$a\,\setminus\,\rho_k,\tilde{\rho_k}$& $\rho_3$ & $\tilde{\rho_3}$ & $\rho_4$&$\tilde{\rho_4}$\\
2.0&7.63&4.99&7.58&4.01\\
\hline
1.5&6.12&3.72&5.60&3.35\\
\hline
1.4&5.71&3.55&5.38&3.27\\
\hline
1.3&5.38&3.41&5.04&3.21\\
\hline
1.2&5.02&3.31&4.70&3.17\\
\hline
1.1&4.56&3.36&4.32&3.15\\
\hline
1.0&3.91&3.20&3.87&3.144\\
\end{tabular}}
\end{table}
}\end{example}

Hence in any minimization algorithm for solving $\mathcal{P}$, 
and given a current iterate $g\in\P[\x]_d$, one may approximate as closely as desired 
the value at $g$ of the objective function as well as its gradient and Hessian
by solving the  semidefinite program (\ref{sdp-2}) for sufficiently large $k$.

\subsection{Concerning the convex cone $C_d(\K)$}

We here assume that the compact (and non necessarily convex) set $\K\subset\R^n$ is a basic semi-algebraic set defined by
\begin{equation}
\label{setk}
\K\,=\,\{\x\in\R^n_+\::\: w_j(\x)\geq0,\:j=1,\ldots,s\},\end{equation}
for some given polynomials $(w_j)\subset\R[\x]$. Denote
by $\Sigma_k\subset\R[\x]_{2k}$ the convex cone of SOS (sum of squares) 
polynomials of degree at most $2k$, and let 
$w_0$ be the constant polynomial equal to $1$, and
$v_j:=\lceil {\rm deg}(w_j)/2\rceil$, $j=0,\ldots,s$.

With $k$ fixed, arbitrary, we now replace the condition
$1-g\in C_d(\K)$ with the stronger condition $1-g\in \mathcal{C}_k\:(\subset C_d(\K))$ where
\begin{equation}
\label{lmi}
\mathcal{C}_k\,=\,\left\{\sum_{j=0}^s \sigma_j\,w_j\::\:\sigma_j\in\Sigma_{k-v_j},\:j=0,1,\ldots,s\,\right\}.
\end{equation}
It turns out that membership in $\mathcal{C}_k$ translates into 
Linear Matrix Inequalities\footnote{A Linear Matrix Inequality (LMI) is a constraint
of the form $\A(\x):=\A_0+\sum_{\ell=1}^t\A_\ell x_\ell\succeq 0$ where each $\A_\ell$, $\ell=0,\ldots,t$, is a real symmetric matrix;
so each entry of the real symmetric matrix $\A(\x)$ is affine in $\x\in\R^t$.
 An LMI always define a convex set, i.e., the set $\{\x\in\R^t:\A(\x)\succeq 0\}$ 
is convex.}  (LMIs) on the coefficients of the polynomials
$g$ and the SOS $\sigma_j$'s; see e.g. \cite{lasserrebook2}. If $\K$ has nonempty interior
then the convex cone $\mathcal{C}_k$ is closed.

\begin{assumption}[Archimedean assumption]
\label{ass-1}
There exist $M>0$ and $k\in\N$ such that the quadratic polynomial $\x\mapsto M-\Vert\x\Vert^2$
belongs to $\mathcal{C}_k$.
\end{assumption}
Notice that Assumption \ref{ass-1} is not restrictive. Indeed, $\K$ being compact,
if one knows an explicit value $\M>0$ such that $\K\subset\{\x:\Vert\x\Vert<M\}$,
then its suffices to add to the definition of $\K$ the redundant quadratic constraint
$w_{s+1}(\x)\geq0$, where $w_{s+1}(\x):=M^2-\Vert\x\Vert^2$.

Under Assumption \ref{ass-1}, $C_d(\K)=\displaystyle\overline{\bigcup_{k=0}^\infty\mathcal{C}_k}$,
that is, the family of convex cones $(\mathcal{C}_k)$, $k\in\N$, provide a converging
sequence of (nested) {\em inner approximations} of the larger convex cone  $C_d(\K)$.

\subsection{A numerical scheme}

In view of the above it is natural to 
consider the following hierarchy of convex optimization problems $(\mathcal{P}_k)$, $k\in\N$,
where for each fixed $k$:
\begin{equation}
\label{sdp-k}
\begin{array}{rl}
\rho_k=\displaystyle\min_{g,\sigma_j}&\displaystyle\int_{\R^n}\exp(-g)\,d\x\\
\mbox{s.t.}&1-g=\displaystyle\sum_{j=0}^s\sigma_jw_j\\
&g_\alpha=0,\quad\forall\vert\alpha\vert<d\\
&g\in\R[\x]_d;\:\sigma_j\in\Sigma_{k-v_j},\:j=0,\ldots,s.
\end{array}\end{equation}
Of course the sequence $(\rho_k)$, $k\in\N$, is monotone non increasing
and $\rho_k\geq\rho$ for all $k$.
Moreover, for each fixed $k\in\N$, $\mathcal{P}_k$ is a convex optimization problem which consists 
of minimizing a strictly convex function under LMI constraints. 

From Corollary \ref{coro1}, $\int_{\R^n}\exp(-g)d\x<\infty$ if and only if $g\in\P[\x]_d$ and so 
the objective function also acts as a barrier for the convex cone $\P[\x]_d$. 
Therefore, to solve $\mathcal{P}_k$ one may use first-order or second-order (local minimization) 
algorithms, starting from an initial guess $g_0\in \P[\x]_d$. At any current iterate 
$g\in\P[\x]_d$ of such an algorithm one may use
the methodology described in \S \ref{gradient-hessian} to approximate the objective function $\int \exp(-g)$ as well as its gradient and Hessian.
Of course as the gradient and Hessian are only approximated, some care is needed to ensure convergence of such an algorithm.
For instance one might try to adapt ideas like the ones described in d'Aspremont \cite{aspremont} 
where for certain optimization problems with noisy gradient information, first-order algorithms with convergence guarantees have been investigated in detail.

\begin{thm}
\label{th-pbk}
Let $\K$ in (\ref{setk}) be compact with nonempty interior and let Assumption \ref{ass-1} hold. Then there exists $k_0$ such that
for every  $k\geq k_0$, problem $\mathcal{P}_k$ in (\ref{sdp-k}) has a unique optimal solution $g^*_k\in\P[\x]_d$.
\end{thm}
\begin{proof}
Firstly, $\mathcal{P}_k$ has a feasible solution for sufficiently large $k$. Indeed consider
the polynomial $\x\mapsto g_0(\x)=\sum_{i=1}^nx_i^{d}$ which belongs to $\P[\x]_d$.
Then as $\K$ is compact, $M-g_0>0$ on $\K$ for some $M$ and so by Putinar's Positivstellensatz \cite{putinar},
$1-g_0/M\in\mathcal{C}_k$ for some $k_0$ (and hence for all $k\geq k_0$). Hence 
$g_0/M$ is a feasible solution for $\mathcal{P}_k$ for all $k\geq k_0$. Of course, as $\mathcal{C}_k\subset C_d(\K)$,
every feasible solution $g\in\P[\x]_d$ satisfies $0\leq g\leq 1$ on $\K$. So proceeding as in the proof of Theorem
\ref{vol-suff-cond} and using the fact that $\mathcal{C}_k$ is closed, the set
\[\{\,g\in\P[\x]_d\cap\mathcal{C}_k:\:\int_{\R^n}\exp(-g)d\x\leq\int_{\R^n}\exp(-\frac{g_0}{M})d\x\,\},\]
is compact. And as the objective function is strictly convex and lower semi-continuous, the optimal solution
$g^*_k\in\P[\x]_d\cap\mathcal{C}_k$ is unique (but the representation of
$1-g^*_k$ in (\ref{sdp-k}) is not unique in general).
\end{proof}

We now consider the asymptotic behavior of the solution of (\ref{sdp-k}) as $k\to\infty$.

\begin{thm}
Let $\K$ in (\ref{setk}) be compact with nonempty interior and let Assumption \ref{ass-1} hold.
If $\rho$ (resp. $\rho_k$) is the optimal value of $\mathcal{P}$ (resp. $\mathcal{P}_k$)
then $\rho=\displaystyle\lim_{k\to\infty}\rho_k$. Moreover, for every $k\geq k_0$, let $g^*_k\in\P[\x]_d$ be the unique 
optimal solution of $\mathcal{P}_k$. Then as $k\to\infty$, $g^*_k\to g^*$ where $g^*$ is the unique optimal solution of $\mathcal{P}$. 
\end{thm}
\vspace{0.2cm}

\begin{proof}
By Theorem \ref{vol-suff-cond}, $\mathcal{P}$ has a unique optimal solution $g^*\in\P[\x]_d$.
Let $\epsilon>0$ be fixed, arbitrary. As $1-g^*\in C_d(\K)$, the polynomial $1-g^*/(1+\epsilon)$
is strictly positive on $\K$, and so by Putinar's Positivstellensatz \cite{putinar},
$1-g^*/(1+\epsilon)$ belongs to $\mathcal{C}_k$ for all
$k\geq k_\epsilon$ for some integer $k_\epsilon$. Hence the polynomial
$g^*/(1+\epsilon)\in\P[\x]_d$ is a feasible solution of $\mathcal{P}_k$ for all $k\geq k_\epsilon$.
Moreover, by homogeneity,
\begin{eqnarray*}
\int_{\R^n}\exp(-\frac{g^*}{1+\epsilon})\,d\x&=&(1+\epsilon)^{n/d}
\int_{\R^n}\exp(-g^*)\,d\x\\
&=&(1+\epsilon)^{n/d}\rho.
\end{eqnarray*}
This shows that $\rho_k\leq (1+\epsilon)^{n/d}\rho$ for all $k\geq k_\epsilon$.
Combining this with $\rho_k\geq\rho$ and the fact that $\epsilon>0$ was arbitrary, yields the 
convergence $\rho_k\to\rho$ as $k\to\infty$.

Next, let $\y\in{\rm int}(C_d(\K)^*)$ be as in the proof of Theorem \ref{vol-suff-cond}.
From $1-g^*_k\in\mathcal{C}_k$ we also obtain $\langle \y,1-g^*_k\rangle\geq0$, i.e.,
\[y_0\geq \langle \y,g^*_k\rangle,\quad \forall k\geq k_0,\]
Recall that the set $\{g\in C_d(\K):\langle \y,g\rangle \leq y_0\}$  is compact. Therefore there exists a subsequence 
$(k_\ell)$, $\ell\in\N$, and $\tilde{g}\in C_d(\K)$ such that
$g^*_{k_\ell}\to \tilde{g}$ as $\ell\to\infty$. In particular, $1-\tilde{g}\in C_d(\K)$ and
$\tilde{g}_\alpha=0$ whenever $\vert\alpha\vert<d$ (i.e., $\tilde{g}$ is homogeneous of degree $d$).
Moreover, one also has the pointwise convergence
$\lim_{\ell\to\infty}g^*_{k_\ell}(\x)=\tilde{g}(\x)$ for all $\x\in\R^n$.
Hence by Fatou's lemma,
\begin{eqnarray*}
\rho\,=\,\lim_{\ell\to\infty}\rho_{k_\ell}&=&\lim_{\ell\to\infty}\int_{\R^n}\exp(-g^*_{k_\ell}(\x))d\x\\
&\geq&\int_{\R^n}\liminf_{\ell\to\infty}\exp(-g^*_{k_\ell}(\x))d\x\\
&=&\int_{\R^n}\exp(-\tilde{g}(\x))d\x\,\geq\,\rho,
\end{eqnarray*}
which proves that $\tilde{g}$ is an optimal solution of $\mathcal{P}$, and by uniqueness
of the optimal solution, 
$\tilde{g}=g^*$. As $(g_{k_\ell})$, $\ell\in\N$, was an arbitrary
converging subsequence, the whole sequence 
$(g^*_k)$ converges to $g^*$.
\end{proof}

\begin{remark}
If desired one may also impose $g$ to be convex (so that $\G_1$ is also convex) by simply requiring
$\z^T\nabla^2g(\x)\z\geq0$ for all $(\x,\z)$. Then one may enforce such a condition by the stronger condition
$(\x,\z)\mapsto \z^T\nabla^2g(\x)\z$ is SOS (i.e., is in $\Sigma[\x,\z]_{d+1}$). Alternatively, if one considers
sets $\G_1\subset\B$ where $\B$ is some sufficient large box containing $\K$, one may also use the 
weaker convexity condition \[\z\nabla^2 g(\x)\z\geq 0\:\mbox{ for all }(\x,\z)\in\B\times\{\,\z: \Vert \z\Vert^2=1\,\}.\]
By using a Putinar positivity certificate the latter also amounts to adding additional LMIs
to problem (\ref{sdp-k}) (which remains convex).
\end{remark}
\vspace{0.2cm}

\section{Conclusion}

We have considered non convex generalizations $\P_0$ and $\P$  of the L\"owner-John ellipsoid problem
where we now look for an homogeneous polynomial $g$ of (even) degree $d>2$. Importantly,
neither $\K$ not the sublevel set $\G_1$ associated with $g$ are required to be convex. However 
both $\P_0$ and $\P$ have an optimal solution (unique for $\P_0$) and a characterization in terms of contact points
in $\K\cap\G_1$ is also obtained as in L\"owner-John's ellipsoid Theorem. Crucial 
is the fact that the Lebesgue volume of $\G_1$ is a strictly convex function of the coefficients of $g$.
This latter fact also permits to define a hierarchy of convex optimization problems 
to approximate as closely as desired the optimal solution of $\P_0$.

\subsection*{Acknowledgement.} This work was partially supported by a grant from the {\it Gaspar Monge Program for Optimization and Operations Research} (PGMO)
of the {\it Fondation Math\'ematique Jacques Hadamard} (France).
\section{Appendix}
\label{proofs}
\subsection{First-order KKT-optimality conditions.} Consider the finite 
dimensional optimization problem:
\[\inf\,\{\,f(\x):\:\A\x\,=\,\b;\:\x \in\,C\,\},\]
for some real matrix $\A\in\R^{m\times n}$, vector $\b\in\R^m$, some closed convex cone $C\subset\R^n$
(with dual cone $C^*=\{\,\y:\y^T\x\geq0,\:\forall\x\,\in C\,\}$) and some convex and differentiable function $f$ with domain $D$.
Suppose that $C$ has a nonempty interior ${\rm int}(C)$ and Slater's condition holds, that is,
there exists $\x_0\in D\cap {\rm int}(C)$ such that $\A\x_0=\b$. The {\it normal cone} at a point $0\neq\x\in C$
is the set $N_C(\x)=\{\y\in C^*:\,\langle \y,\x\rangle=0\}$ (see e.g. \cite[p. 189]{hiriart1}).

Then by Theorem 5.3.3, p. 188 in \cite{hiriart2}, $\x^*\in C$ is an optimal solution if and only if
there exists $(\lambda,\y)\in\R^m\times N_C(\x^*)$ such that:
\begin{equation}
\label{kktconditions}
\A\x^*=\b;\quad \nabla f(\x^*)+\A^T\lambda\,=\,\y\end{equation}
and $\langle \x^*,\y\rangle =0$ follows because $\y\in N_C(\x^*)$.

\subsection{Measures with finite support.}
We restate the following important result stated in \cite[Theorem 1]{kemperman2} and \cite[Theorem 2.1.1, p. 39]{anastassiou}.
\begin{thm}[\cite{anastassiou,kemperman2}]
\label{th-anastassiou}
Let $f_1,\ldots,f_N$ be real-valued Borel measurable functions on a measurable space $\Omega$ and let
$\mu$ be a probability measure on $\Omega$ such that each $f_i$ is integrable with respect to $\mu$. Then there exists
a probability $\nu$ with  finite support in $\Omega$ and such that:
\[\int_\Omega f_i(\x)\,\mu(d\x)\,=\,\int_\Omega f_i(\x)\,\nu(d\x),\quad i=1\ldots,N.\]
One can even attain that the support of $\nu$ has at most $N+1$ points. 
\end{thm}
In fact if $\mathcal{M}(\Omega)_+$ denotes the space of probability measures on $\Omega$, then
the {\it moment space} 
\[Y_N:=\{\y\,=\,\left(\int_\Omega f_k(\x)d\mu(\x)\right),\:k=1,\ldots,N,\quad\mbox{for some }\mu\in \mathcal{M}(\Omega)_+\}\]
is the convex hull of the set $f(\Omega)=\{(f_1(\x),\ldots,f_N(\x)):\:\x\in\Omega\}$ and each point $\y\in Y_N$
can be represented as the convex hull of at most $N+1$ point $f(\x_i)$, $i=1,\ldots,N+1$. (See e.g. 
\S 3, p. 29 in Kemperman \cite{kemperman}.) 

In the proof of Theorem \ref{vol-suff-cond} one uses Theorem
\ref{th-anastassiou} with the $f_i$'s being all monomials $(\x^\alpha)$ of degree equal to $d$ (and so $N={n+d-1\choose d}$). We could also use Tchakaloff's Theorem \cite{bayer} but then we would potentially need 
${n+d\choose d}$ points. An alternative would be to use Tchakaloff's Theorem after ``de-homogenizing"
the measure $\mu$ so that $n$-dimensional moments of order $\vert\alpha\vert=d$ become $(n-1)$-dimensional moments
of order $\vert\alpha\vert\leq d$, and one retrieves the bound ${n-1+d\choose d}$.

\subsection{Proof of Theorem \ref{vol-suff-cond}}
\begin{proof}
(a) As $\mathcal{P}$ is a minimization problem, its feasible set 
$\{\,g\in \H[\x]_{d}:1-g\in C_{d}(\K)\,\}$ can be replaced by the smaller set
\[F\,:=\,\left\{g\in \H[\x]_{d}\::\:\begin{array}{l}\dis\int_{\R^n}\exp(-g(\x))\,d\x\,\leq\,\int_{\R^n}\exp(-g_0(\x))\,d\x\\
1-g\in\,C_{d}(\K)\end{array}\right\},\]
for some $g_0\in\P[\x]_{d}$. Notice that $F\subset\P[\x]_d$ and  $F$ is a closed convex set 
since the convex function $g\mapsto\int_{\R^n}\exp(-g)d\x$ is lower semi-continuous.

Next, let $\z=(z_\alpha)$, $\alpha\in\N^n_{d}$,
be a (fixed) element of ${\rm int}(C_{d}(\K)^*)$ (hence $z_0>0$). By Lemma \ref{empty} such an element exists
and $\langle\z,\g\rangle\geq0$ (as $g\in \P[\x]_d$ is nonnegative). Next there is some $\epsilon>0$ for which $\z\pm\epsilon\,e_\alpha\in C_d(\K)^*$ for every $\alpha$ with $\vert\alpha\vert\leq d$.
Then the constraint $1-g\in C_{d}(\K)$ implies $\langle \z\pm\epsilon\,e_\alpha,1-g\rangle\geq0$ 
(i.e. $\langle \z\pm\epsilon \,e_\alpha,e_0-\g\rangle\geq0$). 
Equivalently $z_0-\langle \z,\g\rangle\geq\epsilon \vert g_\alpha\vert$
for every $\vert\alpha\vert=d$, i.e., $\g$ is bounded and therefore the set $F$ is a compact convex set. Finally, since 
$g\mapsto \int_{\R^n}\exp(-g(\x))d\x$ is strictly convex and lower semi-continuous,
problem $\mathcal{P}$ has a unique optimal solution $g^*\in\P[\x]_d$.

(b) We may and will consider any homogeneous polynomial $g$ as an element of $\R[\x]_{d}$ whose coefficient vector
$\g=(g_\alpha)$ is such that 
$g^*_\alpha=0$ whenever $\vert\alpha\vert<d$. And so Problem $\mathcal{P}$ is equivalent to the problem
\begin{equation}
\label{minvolume1}
\mathcal{P}':\quad \left\{\begin{array}{ll}\rho=\dis\inf_{g\in\R[\x]_{d}}&\dis\int_{\R^n} \exp(-g(\x))\,d\x\\
\mbox{s.t.}&g_\alpha=0,\quad\forall\,\alpha\in\N^n_{d};\:\vert\alpha\vert<d\\
&1-g\,\in\,C_{d}(\K),\end{array}\right.
\end{equation}
where we replaced $g\in\P[\x]_d$ with the equivalent constraints
$g\in\R[\x]_{d}$ and $g_\alpha:=0$ for all $\alpha\in\N^n_{d}$ with $\vert\alpha\vert<d$.
Next, doing the change of variable $h=1-g$, $\mathcal{P}$' reads:
\begin{equation}
\label{minvolume2}
\mathcal{P}':\quad \left\{\begin{array}{ll}\rho=\dis\inf_{h\in\R[\x]_{d}}&\dis\int_{\R^n} \exp(h(\x)-1)\,d\x\\
\mbox{s.t.}&h_\alpha=0,\quad\forall\,\alpha\in\N^n_{d};\:0<\vert\alpha\vert<d\\
&h_0=1\\
&h\,\in\,C_{d}(\K),\end{array}\right.
\end{equation}

As $\K$ is compact, there exists $\theta\in\P[\x]_{d}$ such that
$1-\theta\in{\rm int}(C_{d}(\K))$, i.e., Slater's condition
holds for the convex optimization problem $\mathcal{P}'$. 
Indeed, choose $\x\mapsto\theta(\x):=M^{-1}\Vert\x\Vert^{d}$ for $M>0$ sufficiently large so
that $1-\theta>0$ on $\K$. Hence with $\Vert g\Vert_1$ 
denoting the $\ell_1$-norm of the coefficient vector of $g$ (in $\R[\x]_{d}$),
there exists $\epsilon>0$ such that
for every $h\in B(\theta,\epsilon)(:=\{h\in\R[\x]_{d}:\Vert \theta-h\Vert_1<\epsilon\}$), the polynomial
$1-h$ is (strictly) positive on $\K$. 

Therefore, if $g^*\in{\rm int}(\P[\x]_d)$ the unique optimal solution $(1-g^*)=:h^*\in\R[\x]_{d}$ of $\mathcal{P}$' in
(\ref{minvolume2}) satisfies the Karush-Kuhn-Tucker (KKT) optimality conditions (\ref{kktconditions})
which for problem (\ref{minvolume2}) read:
\begin{eqnarray}
\label{aux60}
\int_{\R^n}\x^\alpha\,\exp(h^*(\x)-1)\,d\x&=&y^*_\alpha,\quad\forall\vert\alpha\vert=d\\
\label{aux61}
 \int_{\R^n}\x^\alpha\,\exp(h^*(\x)-1)\,d\x+\gamma_\alpha&=&y^*_\alpha,\quad\forall\,\vert\alpha\vert<d\\
\label{aux62}
\langle h^*,\y^*\rangle\,=\,0;\quad h^*_0\,=\,1;\:h^*_\alpha&=&0,\quad\forall\,0<\vert\alpha\vert<d
\end{eqnarray}
for some $\y^*=(y^*_\alpha)$, $\alpha\in\N^n_{d}$, in the dual cone $C_{d}(\K)^*\subset\R^{s(d)}$ of $C_{d}(\K)$, and some vector $\gamma=(\gamma_\alpha)$, $0<\vert\alpha\vert<d$. By Lemma \ref{dual-ck},
\[C_{d}(\K)^*\,=\,\{\y\in\R^{s(d)}\::\: \exists\mu\in \mathcal{M}(\K)_+\mbox{ s.t. }y_\alpha=\int_\K\x^\alpha\,d\mu,\:\alpha\in\N^n_{d}\:\},\]
and so (\ref{kkt-suff}) is just (\ref{aux60}) restated in terms of $\mu^*$.

Next, the condition $\langle h^*,\y^*\rangle=0$ (or equivalently, $\langle 1-g^*,\y^*\rangle=0$), reads:
\[\int_{\K} (1-g^*)\,d\mu^*\,=\,0,\]
which combined with $1-g^*\in C_{d}(\K)$ and $\mu^*\in \mathcal{M}(\K)_+$, implies that $\mu^*$ is supported on 
$\K\cap\{\x:g^*(\x)=1\}=\K\cap\G^*_1$. 

Next, let $s:=\sum_{\vert\alpha\vert=d}g^*_\alpha y^*_\alpha \,(=y^*_0)$. From
$\langle 1-g^*,\mu^*\rangle=0$, the measure $s^{-1}\mu^*=:\psi$ is a probability measure
supported on $\K\cap\G^*_1$, and satisfies $\int \x^\alpha d\psi=s^{-1}y^*_\alpha$ for all
$\vert\alpha\vert=d$ (and $\langle 1-g^*,\psi\rangle=0$).

Hence by Theorem \ref{th-anastassiou}  there exists an atomic probability 
measure $\nu^*\in \mathcal{M}(\K\cap\G^*_1)_+$ such
that 
\[\int_{\K\cap\G^*_1}\x^\alpha d\nu^*(\x)\,=\,\int_{\K\cap\G^*_1}\x^\alpha d\psi(\x)\,=\,
s^{-1}\,y^*_\alpha,\qquad\forall\,\vert\alpha\vert=d.\]
In addition $\nu^*$ may be chosen to
be supported on at most $N={n+d-1\choose d}$ points in $\K\cap\G^*_1$ and not $N+1$ points as predicted by 
Theorem \ref{th-anastassiou}. This is because one among the $N$ conditions 
\[\int_{\K\cap\G_1^*}\x^\alpha\,d\nu^*\,=\,s^{-1}\,y_\alpha,\quad\vert\alpha\vert\,=\,d,\]
is redundant as $\langle g^*,\y\rangle =y_0$ and $\nu^*$ is supported on 
$\K\cap\G_1^*$. In other words, $\y$ is not in the interior of the moment space $Y_N$.
Hence in (\ref{kkt-suff}) the measure $\mu^*$ can be substituted with
the atomic measure $s\,\nu^*$ supported on at most $n+d-1\choose d$ contact points
in $\K\cap\G^*_1$.

To obtain $\mu^*(\K)=\frac{n}{d}\int_{\R^n}\exp(-g^*)$, multiply both sides of
(\ref{aux60})-(\ref{aux61}) by $h^*_\alpha$ for every $\alpha\neq0$,  sum up and use $\langle h^*,\y^*\rangle=0$ to obtain
\begin{eqnarray*}
-y^*_0=\sum_{\alpha\neq0} h^*_\alpha\,y^*_\alpha&=&\int_{\R^n}(h^*(\x)-1)\exp(h^*(\x)-1)\,d\x\\
&=&-\int_{\R^n}g^*(\x)\exp(-g^*(\x))\,d\x\\
&=&-\frac{n}{d}\int\exp(-g^*(\x))\,d\x,
\end{eqnarray*}
where we have also used (\ref{euler-1}).

(c) Let $\mu^*:=\sum_{i=1}^s \lambda_i\delta_{\x_i}$ where $\delta_{\x_i}$ is the Dirac measure at the point 
$\x_i\in\K$, $i=1,\ldots,s$. Next, let $y^*_\alpha:=\int \x^\alpha d\mu^*$ for all $\alpha\in\N^n_{d}$, so that
$\y^*\in C_d(\K)^*$. In particular $\y^*$ and $g^*$ satisfy
\[\langle 1-g^*,\y^*\rangle\,=\,\int_\K( 1-g^*)d\mu^*\,=\,0,\] 
because $g^*(\x_i)=1$ for all $i=1,\ldots,s$.
In other words, the pair $(g^*,\y^*)$ satisfies the KKT-optimality conditions
associated with the convex problem $\mathcal{P}$. But since Slater's condition holds for $\mathcal{P}$,
those conditions are also sufficient for $g^*$ to be an optimal solution of $\mathcal{P}$, the desired result
\end{proof}

\subsection{Proof of Theorem \ref{vol-suff-cond-a}}
\label{conda}
\begin{proof}
First observe that (\ref{minvolume-a}) reads
\begin{equation}
\label{minvolume-aa}
\mathcal{P}:\quad\min_{\a\in\R^n}\:\left\{\min_{g\in\P[\x]_d}\:\{{\rm vol}(\G^\a_1)\::\:1-g_\a\in C_d(\K)\}\right\},\end{equation}
and  notice that the constraint 
$1-g_\a\in C(\K)$ is the same as $1-g\in C(\K-\a)$. And so
for every $\a\in\R^n$,  the inner minimization problem 
\[\min_{g\in\P[\x]_d}\:\{{\rm vol}(\G^\a_1)\::\:1-g_\a\in C_d(\K)\}\]
of (\ref{minvolume-aa}) reads\
\begin{equation}
\label{inner2}
\rho_\a\,=\,\min_{g\in\P[\x]_d}\:\{{\rm vol}(\G_1)\::\:1-g\in C_d(\K-\a)\}.
\end{equation}
From Theorem \ref{vol-suff-cond} (with $\K-\a$ in lieu of $\K$), problem (\ref{inner2})
has a unique minimizer $g^\a\in\P[\x]_d$ with value 
$\rho_\a=\int_{\R^n}\exp(-g^\a)d\x=\int_{\R^n}\exp(-g^\a_\a)d\x$. \\

Therefore, in a minimizing sequence
 $(\a_\ell,g^{\a_\ell})\subset \R^n\times\P[\x]_d$, $\ell\in\N$, for problem $\mathcal{P}$ in (\ref{minvolume-a}) with 
 \[\rho=\lim_{\ell\to\infty}\:\int_{\R^n}\exp(-g^{\a_\ell})d\x,\]
we may and will consider that for every $\ell$, the homogeneous polynomial 
$g^{\a_\ell}\in\P[\x]_d$) solves the inner minimization problem (\ref{inner2})
with $\a_\ell$ fixed. For simplicity of notation rename  $g^{\a_\ell}$ as $g^\ell$ and 
$g^{\a_\ell}_{\a_\ell}$ ($=g^{\a_\ell}(\x-\a_\ell)$) as $g^\ell_{\a_\ell}$.

As observed in the proof of Theorem \ref{vol-suff-cond}, there is $\z\in{\rm int}(C_d(\K)^*)$ such that
 $\langle 1-g^\ell_{\a_\ell},\z\rangle \geq0$ and by Corollary I.1.6 in Faraut et Kor\'anyi \cite{faraut}, the set $\{h\in C_{d}(\K):\langle \z,h\rangle\leq z_0\}$ is compact. 
 
 Also, $\a_\ell$ can be chosen with $\Vert\a_\ell\Vert\leq M$ for all $\ell$ (and some $M$), otherwise 
the constraint $1-g_{\a_\ell}\in C_d(\K)$ would impose a much too large volume ${\rm vol}(\G^{\a_\ell}_1)$. 

Therefore, there is a subsequence $(\ell_k)$, $k\in\N$, and a point $(\a^*,\theta^*)\in\R^n\times C_d(\K)$
such that
\[\lim_{k\to\infty}\a_{\ell_k}\,=\,\a^*;\qquad \lim_{k\to\infty}(g^{\ell_k}_{\a_{\ell_k}})_\alpha\,=\,\theta^*_\alpha,\quad\forall
\alpha\in\N^n_d.\]
Recall the definition (\ref{def}) of $g^\ell_{\a_\ell}(\x)=g^\ell(\x-\a_\ell)$ for the homogeneous
polynomial $g^\ell\in\P[\x]_d$ with coefficient vector $\g^\ell$, i.e.,
\[(g^\ell_{\a_{\ell}})_\alpha\,=\,p_\alpha(\a_\ell,\g^\ell),\qquad\forall\alpha\in\N^n_d,\]
for some polynomials $(p_\alpha)\subset\R[\x,\g]$, $\alpha\in\N^n_d$. In particular, for every $\alpha\in\N^n_d$ with $\vert\alpha\vert=d$,
$p_\alpha(\a_\ell,\g^\ell)=(g^\ell)_\alpha$. And so 
for every $\alpha\in\N^n_d$ with $\vert\alpha\vert=d$,
\[\theta^*_\alpha\,=\,\lim_{k\to\infty}\,=\,(g^{\ell_k})_\alpha.\]
If we define the homogeneous polynomial $g^*$ of degree $d$ by $(g^*)_\alpha=\theta^*_\alpha$ for every
$\alpha\in\N^n_d$ with $\vert\alpha\vert=d$, then
\[\lim_{k\to\infty}(g^{\ell_k}_{\a_{\ell_k}})_\alpha\,=\,\lim_{k\to\infty}p_\alpha(\a_{\ell_k},\g^{\ell_k}),
\,=\,p_\alpha(\a^*,\g^*),\quad\forall\alpha\in\N^n_d.\]
 This means that for every $\alpha\in\N^n_d$,
\[\theta^*(\x)\,=\,g^*(\x-\a^*),\quad\x\in\R^n.\]
In addition, as $\g^{\ell_k}\to\g^*$ as $k\to\infty$, one has the pointwise convergence $g^{\ell_k}(\x)\to g^*(\x)$ for all $\x\in\R^n$.
Therefore, by Fatou's Lemma (see e.g. Ash \cite{ash}),
\[\rho\,=\,\lim_{k\to\infty}\:\int_{\R^n}\exp(-g^{\ell_k})\,d\x\,\geq\,
\int_{\R^n}\liminf_{k\to\infty}\exp(-g^{\ell_k})\,d\x\,=\,\int_{\R^n}\exp(-g^*)\,d\x,\]
which proves that $(\a^*,g^*)$ is an optimal solution of (\ref{minvolume-a}). 

In addition $g^*\in\P[\x]_d$ is an optimal solution of the inner minimization problem
in (\ref{inner2}) with $\a:=\a^*$. Otherwise an optimal solution $h\in\P[\x]_d$ of (\ref{inner2}) with $\a=\a^*$ would yield a solution $(\a^*,h)$ with associated cost $\int_{\R^n}\exp(-h)$ strictly smaller than $\rho$, a contradiction.

Hence by Theorem \ref{vol-suff-cond} (applied to problem (\ref{inner2})), if $g^*\in{\rm int}(\P[\x]_d)$
there is a finite Borel measure
$\mu^*\in \mathcal{M}(\K-\a^*)_+$ such that
\begin{eqnarray*}
\int_{\R^n}\x^\alpha\exp(-g^*)d\x&=&\int_{\K-\a^*}\x^\alpha\,d\mu^*,\qquad\forall\vert\alpha\vert=d\\
\int_{\K-\a^*}(1-g^*)\,d\mu^*&=&0;\quad \mu(\K-\a^*)=\frac{n}{d}\int_{\R^n}\exp(-g^*)\,d\x.
\end{eqnarray*}
And so $\mu^*$ is supported on the set 
\[V=\{\,\x\in\K-\a^*:\:g^*(\x)=1\}\,=\,\{\,\x\in\K:\:g^*(\x-\a^*)=1\,\}\,=\,\K\cap \G^{\a^*}_1.\]
Invoking again \cite[Theorem 2.1.1, p. 39]{anastassiou}, there exists an atomic measure $\nu^*\in \mathcal{M}(\K-\a^*)_+$ supported on 
at most ${n-1+d\choose d}$ of $\K-\a^*$ with same moments of order $d$ as $\mu^*$.
\end{proof}


\begin{thebibliography}{las}
\bibitem{anastassiou}
G.A. Anastassiou; {\em Moments in Probability and Approximation Theory}, Longman Scientific \& Technical, UK, 1993.
\bibitem{aspremont}
A. d'Aspremont. Smooth optimization with approximate gradient. {\it SIAM J. Optim.} {\bf 19} (2008), pp. 1171--1183.
\bibitem{ash}
R.B. Ash. {\em Real Analysis and Probability}, Academic Press Inc., Boston, 1972.
\bibitem{ball1}
K. Ball. Ellipsoids of maximal volume in convex bodies, {\it Geom. Dedicata} {\bf 41} (1992), pp. 241--250.
\bibitem{ball2}
K. Ball. Convex geometry and functional analysis, In {\em Handbook of the Geometry of Banach Spaces I}, W.B. Johnson and J. Lindenstrauss (Eds.), North Holland, Amsterdam 2001, pp. 161--194.
\bibitem{barvinok}
A.I. Barvinok. Computing the volume, counting integral points, and exponential sums. {\it Discrete \& Comput. Geom.} {\bf 10} (1993), pp. 123--141. 
\bibitem{bastero}
J. Bastero and M. Romance. John's decomposition of the identity in the non-convex case,
{\it Positivity} {\bf 6} (2002), pp. 1--16.
\bibitem{bayer}
C. Bayer and J. Teichmann. The proof of Tchakaloff's theorem,
{\em Proc. Amer. Math. Soc.} {\bf 134} (2006), pp. 303--3040.
\bibitem{bookstein}
F.L. Bookstein. Fitting conic sections to scattered data, {\it Comp. Graph. Image. Process.} {\bf 9} (1979), pp. 56--71.
\bibitem{calafiore}
G. Calafiore. Approximation of $n$-dimnsional data using spherical and ellipsoidal primitives,
{\it IEEE Trans. Syst. Man. Cyb.} {\bf 32} (2002), pp. 269--276.
\bibitem{chernousko}
F.L. Chernousko. Guaranteed estimates of undetermined quantities by means of ellipsoids,
{\it Sov. Math. Dodkl.} {\bf 21} (1980), pp. 396--399.
\bibitem{croux}
C. Croux, G. Haesbroeck and P.J. Rousseeuw. Location adjustment for the minimum volume ellipsoid estimator,
{\it Stat. Comput.} {\bf 12} (2002), pp. 191--200.
\bibitem{giannopoulos}
A. Giannopoulos, I. Perissinaki and A. Tsolomitis. A. John's theorem for an arbitrary pair of convex bodies. {\it Geom. Dedicata}  {\bf 84} (2001), pp. 63--79. 
\bibitem{dyer}
M.E. Dyer, A.M. Frieze and R. Kannan. A random polynomial-time algorithm for approximating the volume of convex bodies.
{\it J. ACM} {\bf 38} (1991), pp. 1--17.
\bibitem{faraut}
J. Faraut and A. Kor\'anyi. {\it Analysis on Symmetric Cones}, Clarendon Press, Oxford, 1994.
\bibitem{freitag}
E. Freitag and R. Busam. {\em Complex Analysis}, Second Edition, Springer-Verlag, Berlin, 2009.
\bibitem{gander}
W. Gander, G.H. Golub and R. Strebel. 
Least-squares fitting of circles and ellipses, {\it BIT} {\bf 34} (1994), pp. 558--578.
\bibitem{nie-helton}
J.W. Helton and J. Nie.  A semidefinite approach for truncated K-moment problems,
{\it Fond. Comput. Math.} {\bf 12} (2012), pp. 851--881.
\bibitem{henk}
M. Henk. L\"owner-John ellipsoids, {\it Documenta Math.} (2012), Extra volume: Optimization Stories, pp. 95--106.
\bibitem{gloptipoly}
D. Henrion, J.B. Lasserre and J. Lofberg. Gloptipoly 3: moments, optimization and semidefinite programming,
{\em Optim. Methods and Softwares} {\bf 24} (2009),   pp. 761--779. 
\bibitem{sirev}
D. Henrion, J.B. Lasserre and  C. Savorgnan. Approximate volume and integration of basic semi-algebraic sets,
SIAM Review {\bf 51} (2009),   pp. 722--743.  
\bibitem{ellipsoids}
D. Henrion, D. Peaucelle, D. Arzelier and M.Sebek. Ellipsoidal approximation of the stability domain of a polynomial,
{\it IEEE Trans. Aut. Control} {\bf 48} (2003), pp. 2255--2259.
 \bibitem{ieee}
D. Henrion and J.B. Lasserre. Inner approximations for polynomial matrix inequalities and robust stability regions,
{\it IEEE Trans. Aut. Control} {\bf 57} (2012),  pp. 1456--1467. 
\bibitem{lmi1}
D. Henrion, M. Sebek and V. Kucera.
Positive polynomials and and robust stabilization with fixed-order controllers,
{\it IEEE Trans. Aut. Control} {\bf 48} (2003),  pp. 1178--1186.
\bibitem{hiriart1}
J.B. Hiriart-Urruty and C. Lemarechal. {\em Convex Analysis and Minimization Algorithms I},
Springer-Verlig, Berlin, 1993.
\bibitem{hiriart2}
J.B. Hiriart-Urruty and C. Lemarechal. {\em Convex Analysis and Minimization Algorithms II},
Springer-Verlig, Berlin, 1993.
\bibitem{lmi2}
A. Karimi, H. Khatibi and R. Longchamp. Robust control of polytopic systems by convex optimization,
{\it Automatica} {\bf 43} (2007), pp. 1395--1402.
\bibitem{kemperman}
J.H.B. Kemperman. Geometry of the moment problem, in {\it Moments in Mathematics},
H.J. Landau (Ed.), {\it Proc. Symposia in Applied Mathematics} {\bf 37} (1987), pp. 16--53.
\bibitem{kemperman2}
J.H.B. Kemperman. The general moment problem, a geometric approach,
{\it Annals Math. Stat.} {\bf 39} (1968), pp. 93--122.
\bibitem{lasserresiopt}
J.B. Lasserre. Global optimization with polynomials and the problem of moments,
{\it SIAM J. Optim.} {\bf 11} (2001), pp. 796--817.
\bibitem{lasserrebook2}
J.B. Lasserre. {\em Moments, Positive Polynomials and Their Applications}, Imperial College, London, 2009.
\bibitem{lasserre-dcg}
J.B. Lasserre. Recovering an homogeneous polynomial from moments of its level set, Discrete \& Comput. Geom. {\bf 50} (2013), pp. 673--678.
\bibitem{morosov1}
A. Morosov and S. Shakirov. New and old results in resultant theory,
Theor. Math. Physics {\bf 163} (2010), pp. 587--617.
\bibitem{morosov2}
A. Morosov and S. Shakirov. Introduction to integral discriminants, J. High Energy Phys. {\bf 12} (2009), {\tt arXiv:0911.5278v1}, 2009.
\bibitem{polytopes}
U. Nurges. Robust pole assignment
via reflection coefficientsof polynomials, {\it Automatica} {\bf 42} (2006), pp. 1223--1230.
\bibitem{orourke}
J. O'Rourke and N.I. Badler. Decomposition of three-dimensional objets into spheres, {\it IEEE Trans. 
Pattern Anal. Machine Intell.} {\bf 1} (1979), pp. 295--305.
\bibitem{pratt}
V. Pratt. Direct least squares fittingof algebraic surfaces,
{\it ACM J. Comp. Graph.} {\bf 21} (1987).
\bibitem{putinar}
M. Putinar. Positive polynomials on compact semi-algebraic sets, Indiana Univ. Math. J. {\bf 42} (1993), pp. 969--984
\bibitem{rockafellar}
R.T. Rockafellar. {\em Convex Analysis}, Princeton University Press, Princeton, New Jersey, 1970.
\bibitem{rosen}
J.B. Rosen. Pattern separation by convex programming techniques, {\it J. Math. Anal. Appl.} {\bf 10} (1965), pp. 123--1324.
\bibitem{rosin}
P.L. Rosin. A note on the least squares fitting of ellipses, {\it Pattern Recog. Letters} {\bf 14} (1993), pp. 799--808.
\bibitem{rosin2}
P.L. Rosin and G.A. West. Nonparametric segmentation of curves into various representations,
{\it IEEE Trans. Pattern Anal. Machine Intell.} {\bf 17} (1995), pp. 1140--1153.
\bibitem{rousseeuw}
P.J. Rousseeuw and A.M. Leroy. {\it Robust Regression and Outlier Detection}, John Wiley, New York (1987).
\bibitem{royden}
H.L. Royden. {\it Real Analysis}, Macmillan, 1968.
\bibitem{sun}
P. Sun and R. Freund. Computation of minimum-volume covering ellipsoids, {\it Oper. Res.} {\bf 52} (2004), pp. 690--706.
\bibitem{taubin}
G. Taubin. Estimation of planar curves, surfaces and nonplanar space curves defined by implicit equations, with applications to to edge and range image segmentation, {\it IEEE Trans. Pattern Anal. Machine Intell.} {\bf 13} (1991), pp. 1115--1138.
\bibitem{boyd}
L. Vandenberghe and S. Boyd. Semidefinite programming, {\it SIAM Rev.} {\bf 38} (1996), pp. 49--95.
\bibitem{widder}
D.V. Widder. {\it The Laplace Transform}, Princeton University Press, Princeton, 1946.
\end{thebibliography}
\end{document}